\documentclass[a4paper,12pt]{scrartcl}

\usepackage[bookmarks,pagebackref]{hyperref}
\usepackage{amsrefs}
\usepackage[utf8]{inputenc}
\usepackage[english]{babel}
\usepackage{amssymb}
\usepackage{amsmath}
\usepackage{latexsym}
\usepackage{amsthm}
\usepackage{eucal}
\usepackage{bbm}
\usepackage{chngcntr}
\usepackage{apptools}
\usepackage{mathtools}
\usepackage{authblk}
\usepackage[pdflatex]{crop}
\usepackage{tikz}
\usepackage{tikz-cd}
\usepackage{authblk}
\usepackage{microtype}
\usepackage{enumitem}
 
\usepackage[nottoc]{tocbibind}
\allowdisplaybreaks

\setkomafont{disposition}{\normalfont\bfseries}
\newcommand{\auth}[0]{{Tobias Fritz and Paolo Perrone}}
\newcommand{\tit}[0]{{A Probability Monad as the Colimit of Spaces of Finite Samples}}
\newcommand{\kw}[0]{{Categorical probability theory, Giry monad, optimal transport, Kantorovich-Rubinstein distance, monoidal Kan extension}}

\hypersetup{
pdfauthor={\auth},%
pdftitle={\tit},%
colorlinks, linktocpage=true, pdfstartpage=1, pdfstartview=FitV,%
breaklinks=true, pdfpagemode=UseNone, pageanchor=true, pdfpagemode=UseOutlines,%
plainpages=false, bookmarksnumbered, bookmarksopen=true, bookmarksopenlevel=1,%
hypertexnames=true, pdfhighlight=/O,%
urlcolor=black, linkcolor=black, citecolor=black, 
}
\numberwithin{equation}{section}

\theoremstyle{plain}
\newtheorem{thm}{Theorem}[subsection]
\AtAppendix{\counterwithin{thm}{section}}
\newtheorem{lemma}[thm]{Lemma}
\newtheorem{prop}[thm]{Proposition}
\newtheorem{cor}[thm]{Corollary}

\newtheorem{deph}[thm]{Definition}

\theoremstyle{definition}
\newtheorem{remark}[thm]{Remark}

\newcommand{\N}{\mathbb{N}}

\newcommand{\R}{\mathbb{R}}

\newcommand{\cat}[1]{{\mathsf{#1}}} 
\newcommand{\ar}[2][]{\arrow{#2}{#1}}
\newcommand{\uni}[2][]{\arrow[dashrightarrow]{#2}{#1}} 
\newcommand{\nat}[2][]{\arrow[Rightarrow]{#2}{#1}} 
\newcommand{\idar}[2][]{\arrow[equal]{#2}{#1}} 
\DeclareMathOperator{\1}{\mathbbm{1}}

\DeclareMathOperator*{\colim}{colim}

\newcommand{\Lip}{\mathrm{Lip}}

\newcommand{\E}{\mathbb{E}}

\newcommand{\diam}{\mathrm{diam}}

\DeclareMathOperator{\e}{\varepsilon}

\newcommand{\op}{\mathrm{op}}

\let\originalleft\left
\let\originalright\right
\renewcommand{\left}{\mathopen{}\mathclose\bgroup\originalleft}
\renewcommand{\right}{\aftergroup\egroup\originalright}

\title{\tit}
\author[1]{Tobias Fritz\footnote{Correspondence: tfritz [at] perimeterinstitute.ca}}
\author[2]{Paolo Perrone\footnote{Correspondence: perrone [at] mis.mpg.de}}

\affil[1]{\small Perimeter Institute for Theoretical Physics, Waterloo, Canada}
\affil[2]{\small Max Planck Institute for Mathematics in the Sciences, Leipzig, Germany}
\date{}

\begin{document}

\maketitle

\begin{abstract}
\addcontentsline{toc}{section}{Abstract}
We define and study a probability monad on the category of complete metric spaces and short maps. It assigns to each space the space of Radon probability measures on it with finite first moment, equipped with the Kantorovich--Wasserstein distance. This monad is analogous to the Giry monad on the category of Polish spaces, and it extends a construction due to van Breugel for compact and for 1-bounded complete metric spaces.
 
We prove that this \emph{Kantorovich monad} arises from a colimit construction on finite power-like constructions, which formalizes the intuition that probability measures are limits of finite samples. The proof relies on a criterion for when an ordinary left Kan extension of lax monoidal functors is a monoidal Kan extension. 
The colimit characterization allows the development of integration theory and the treatment of measures on spaces of measures, without measure theory.
 
We also show that the category of algebras of the Kantorovich monad is equivalent to the category of closed convex subsets of Banach spaces with short affine maps as morphisms.
\end{abstract}

\newpage
\tableofcontents

\section{Introduction}

In existing categorical approaches to probability theory, one works with a suitable category of measurable spaces and equips it with a monad, which associates to every space $X$ the space of probability measures on $X$~\cite{early}. This applies e.g.~to the Giry monad~\cite{giry} and to the Radon monad, discovered in~\cites{semadeni,swirszcz} and named such in~\cite{furber-jacobs}; see also~\cite{jacobs} for a general overview of probability monads and~\cite{lucyshyn-wright} for a more general setup. These monads constitute an additional piece of structure on the underlying category. Here, we introduce another such monad---the \emph{Kantorovich monad}---which lives on the category of complete metric spaces and extends the analogous monads studied by van Breugel~\cite{breugel} on the full subcategories of compact metric spaces and of 1-bounded complete metric spaces. An extension to suitable non-bounded spaces is necessary for our goal to redevelop basic probability theory categorically, because generic distributions of random variables in probability theory may not have bounded support -- the Gaussian is a prominent example.

We prove that the monad structure of the Kantorovich monad naturally arises from a colimit construction on the underlying category, which is motivated by the operational interpretation of a probability measure as a formal limit of finite samples. This allows to approach some elements of probability measure theory, such as integration, in terms of simpler considerations based on finite sets.
Among other benefits, we hope that this may help to make probability theory more constructive, perhaps in a way that allows for straightforward implementation in a functional programming language or proof assistant.

Besides this colimit characterization, another reason for using probability monads on \emph{metric} spaces is the following. Most results in probability theory are concerned with approximations (in some sense or another), often in a quantitative manner. Therefore we expect that working with metric spaces will allow us to find categorical formulations, proofs, or perhaps even \emph{generalizations} of such approximation results, such as the law of large numbers, or the Glivenko-Cantelli theorem on the convergence of empirical distributions.

In algebra, theoretical computer science, and other fields, monads often arise from equational theories \cite{comp-monads}. Categorically, this is formalized by presenting a monad in terms of an associated Lawvere theory, operad, or generalized operad, via a suitable coend or more general colimit. From this perspective, our colimit characterization formalizes the idea that this probability monad models a kind of algebraic theory presented by the operations of taking convex combinations with uniform weights. However, our way of presenting the Kantorovich monad does not involve a Lawvere theory or an operad, but rather a \emph{graded monad}~\cite{ssm}. 

This theme has also been pursued in the work of van Breugel on the Kantorovich monad for 1-bounded complete metric spaces~\cite{breugel}, in particular with the consideration of metric mean-value algebras~\cite[Definition~6]{BHMW}. A similar idea underlies recent ongoing work of Mardare, Panangaden and Plotkin. In~\cite[Theorem~10.9]{mpp}\footnote{\label{mpp_error} However, in order for their Theorem 10.9 to be correct, their definition of the $p$-Wasserstein space $\Delta[M]$ needs to be restricted to the measures of finite $p$-th moment, as we do in Section~\ref{repth} for $p=1$. The reason is that a probability measure of infinite $p$-th moment has infinite $p$-Wasserstein distance from any finitely supported probability measure, e.g.~because it has infinite distance from any Dirac measure. The error in the proof (as pointed out to us by Prakash Panangaden) is in the claim that the $p$-Wasserstein distance metrizes weak convergence, which is true only for finite $p$-th moment.}, they consider the underlying functor of the Kantorovich monad on complete separable metric spaces as the free algebras of an $\R_+$-enriched Lawvere theory, which is closely related to our Theorem~\ref{Palgthm}\ref{csmet}.

\paragraph{Summary.} 
In Section \ref{secwspaces} we introduce the main mathematical constructions that we use in this work: the categories $\cat{Met}$ and $\cat{CMet}$ of (complete) metric spaces and short maps, and the Radon measures on such spaces with finite first moment. We prove (Theorem \ref{fl-complete}) that such measures are equivalently linear, positive and $\tau$-smooth functionals on the space of Lipschitz functions. In Section~\ref{conwsp} we introduce the Wasserstein metric, and show the functoriality of the Wasserstein space construction (Lemma \ref{Pfunctor}), resulting in the \emph{Kantorovich functor} $P$.

In Section \ref{seccolimit} we prove (Theorem \ref{pcolimit} and Corollary \ref{empdcolimit}) that the Wasserstein spaces and the Kantorovich functor can be obtained as the colimit of the \emph{power functors}, defined in \ref{powerfunsec}, and that the universal arrow is given by the empirical distribution map, which we define in \ref{empdist}. 

In Section \ref{secmonad} we prove that $P$ has a monad structure (Theorem \ref{monadth}), which arises naturally from the colimit characterization, given the particular monoidal structure of the power functors (Theorems \ref{monoidalfunctorunsym} and \ref{monoidalfunctorsym}). This can be interpreted as a Kan extension in the 2-category $\cat{MonCat}$ of monoidal categories and lax monoidal functors (Theorem \ref{FinUniftoNmonad}). 

In Section \ref{Palgssec} we study the algebras of $P$. We show (Theorem \ref{Palgthm}) that the algebras are convex metric spaces whose convex structure is compatible with the metric. This implies in turn that the algebras are equivalently closed convex subsets of Banach spaces (Theorem \ref{algban}).

In Appendix \ref{colims_met_app} we study colimits in the category of metric spaces, and prove (Proposition~\ref{tensor_cocont}) that the tensor product preserves colimits. This is used in Section \ref{secmonad} to define the monad structure of $P$.

In Appendix \ref{kan} we give the 2-categorical details (Theorem \ref{month}) of why the Kan extensions used in this paper are Kan extensions in $\cat{MonCat}$, or \emph{algebraic Kan extensions}.

In Appendix \ref{ccms} we define the operad of convex spaces, and show that $\cat{Met}$ is a pseudoalgebra for this operad, giving further motivation for the power functor construction. We also show that, in agreement with the \emph{microcosm principle}, $P$-algebras in the form of convex spaces with metric compatibility are particular internal algebras in $\cat{Met}$, i.e.~they form a full subcategory.

\paragraph{Acknowledgments.} We thank Rory Lucyshyn-Wright, Rostislav Matveev, Prakash Panangaden, Sharwin Rezagholi, Patrick Schultz and David Spivak for helpful discussions. We also thank Prakash Panangaden and David Spivak for helpful comments on a draft of this work, and Roald Koudenburg and Mark Weber for pointers to the literature.
Moreover, we thank the anonymous reviewer for very helpful comments.

\section{Wasserstein spaces}
\label{secwspaces}

\subsection{Categorical setting}
\label{catset}

Two categories are of primary interest to us. The first one is the monoidal category $\cat{Met}$, where:
\begin{itemize}
 \item Objects are metric spaces, with the classical notion of metric as a distance function $d : X\times X \to [0,\infty)$ satisfying identity of indiscernibles, symmetry, and the triangle inequality;
 \item Morphisms are short maps (also called 1-Lipschitz maps), i.e.~functions $f:X\to Y$ such that for all $x,x'\in X$:
 \begin{equation}
  d_Y(f(x),f(x')) \le d_X(x,x')\;; 
 \end{equation}
 \item As monoidal structure, we define $X\otimes Y$ to be the set $X\times Y$, equipped with the $\ell^1$-product metric:
 \begin{equation}
  d_{X\otimes Y} \big( (x,y), (x',y') \big) := d_X(x,x') + d_Y(y,y').
 \end{equation}
\end{itemize}
The second one is its full subcategory $\cat{CMet}$ whose objects are complete metric spaces.

It is useful to think of metric spaces and short maps as enriched categories and enriched functors~\cites{lawvere,seriously}, although in that context one typically works with a more relaxed notion of metric space, such as allowing infinite distances (or even Lawvere metric spaces), in order to guarantee cocompleteness. For $\cat{Met}$ and $\cat{CMet}$, we investigate the existence of particular colimits and their preservation by the monoidal product in Appendix~\ref{colims_met_app}.

In our considerations, the concept of \emph{isometric embedding} will often come up. It is worth noting that together with the bijective short maps, the isometric embeddings in $\cat{Met}$ form an orthogonal factorization system, which is, in many respects, analogous to the (bo,ff) factorization system on $\cat{Cat}$. Isometric embeddings are very close to being characterized 1-categorically as the extremal monomorphisms: every extremal monomorphism in $\cat{Met}$ is an isometric embedding, but only the isometric embeddings of \emph{closed} subspaces are extremal monomorphisms, since the isometric embedding of a subspace into its closure is an epimorphism. Since a subspace of a complete metric space is complete if and only if it is closed, it follows that in $\cat{CMet}$ the class of extremal monomorphisms coincides with the isometric embeddings.

\subsection{Analytic setting}
\label{anaset}

The following definitions of an analytic nature will only be needed in this section and in Section~\ref{seccolimit}, where we prove our colimit characterization; all subsequent developments will use the latter and therefore do not require any measure theory. 

Every metric space is a topological space, and so also a measurable space with the its $\sigma$-algebra. We will always suppose probability measures to be Borel, and Radon, i.e.~inner regular with respect to compacts.

For $X\in\cat{Met}$, we write $\Lip(X)$ for the space of Lipschitz functions $X\to\R$, where $\R$ carries its usual Euclidean metric. Every Lipschitz function is a scalar multiple of an element of $\cat{Met}(X,\R)$, i.e.~of a short map $X\to\R$. We expect that working with the latter space, or even just with $\cat{Met}(X,\R_+)$, would be useful for achieving further abstraction. However, currently we prefer to work with $\Lip(X)$, which has the added benefit of being a vector space.

\subsection{Finite first moments and a representation theorem}
\label{repth}

In order to define Wasserstein spaces, we first have to define probability measures of finite first moment, which are those for which every Lipschitz function has an expectation value.

\begin{deph}\label{ffm}
 Let $X\in\cat{Met}$ and $p$ be a probability measure on $X$. We say that $p$ has \emph{finite first moment} if the expected distance between two random points is finite, i.e.~if
\[
	\int d(x,y) \, dp(x) \, dp(y) < +\infty.
\]
\end{deph}

We have borrowed this elegant formulation from Goubault-Larrecq~\cite[Section~1]{gl}, who attributes it to Fernique.

\begin{lemma}
\label{ffmchar}
The following statements are equivalent for a probability measure $p$ on $X\in\cat{CMet}$:
\begin{enumerate}
\item\label{pffm} $p$ has finite first moment.
\item\label{existsx0} There is $y\in X$ such that the expected distance from $y$ is finite,
\begin{equation*}
\int d(y,x)\,dp(x) < +\infty.
\end{equation*}
\item\label{allx0} For all $z\in X$, the expected distance from $z$ is finite,
\begin{equation*}
\int d(z,x)\,dp(x) < +\infty.
\end{equation*}
\item\label{fexpect} Every $f\in \Lip(X)$ has finite expectation value,
 \begin{equation*}
 \label{fint}
  \int f(x)\,dp(x) < +\infty.
 \end{equation*}
\end{enumerate}
\end{lemma}

\begin{proof}
Since $p$ is a probability measure, we know that $X$ is nonempty and thus we can always choose a point whenever we need one.

\begin{itemize}
 \item \ref{pffm}$\Rightarrow$\ref{existsx0}: if the integral of a nonnegative function is finite, then the integrand is finite at (at least) one point.
 \item \ref{existsx0}$\Rightarrow$\ref{allx0}: For all $z\in X$, and for $y$ as in~\ref{existsx0}, we have:
 \begin{align*}
  \int d(z,x)\,dp(x) &\le \int \big( d(z,y) + d(y,x) \big) \,dp(x) \\
   &= d(z,y) + \int d(y,x) \,dp(x), 
 \end{align*}
where the first term is finite for every $z$, and the second term is finite by hypothesis.
 \item \ref{allx0}$\Rightarrow$\ref{fexpect}: Since $f$ is integrable if and only if $|f|$ is, it is enough to consider the case $f\geq 0$. Then for an arbitrary $z\in X$,
 \begin{align*}
 \int f(x)\,dp(x) &= \int \left( f(x) - f(z) + f(z)\right) \,dp(x) \\
 &\le f(z) + \int |f(x) - f(z)| \, dp(x) \\
 &\le f(z) + L_f \int  d(x,z) \, dp(x) < +\infty,
 \end{align*}
 where $L_f$ is the Lipschitz constant of $f$, which is a finite number.
 \item \ref{fexpect}$\Rightarrow$\ref{pffm}: Since the distance is short in each argument, the function
 \begin{equation*}
  x \mapsto \int_X d(x,y)\,dp(y)
 \end{equation*}
 is finite by assumption and automatically short. Therefore its expectation is again finite by hypothesis, which implies the finite first moment condition.\qedhere
\end{itemize}
\end{proof}

From now on, we write $PX$ for the set of probability measures on $X$ with finite first moment. Later, we will equip this set with a metric. As we discuss in more detail below, pushing forward measures along a short map $f:X\to Y$ defines a function $Pf:PX\to PY$ which makes $P$ into a functor.

Often measures are specified by how they act on functions by integration, such as in the definition of the Daniell integral or in the Riesz representation theorem. We will now state an analogous result for $PX$. Concretely, every $p\in PX$ defines a linear functional $\E_p:\Lip(X)\to\R$ given by mapping every function to its expectation value,
\begin{equation}
\label{fEp}
f\longmapsto \E_p(f) := \int f(x)\, dp(x).
\end{equation}
We can thus consider $\E$ as a map $\E:PX\to\Lip(X)^*$ into the algebraic dual. Each functional $\E_p$ has a number of characteristic properties: it is linear, positive, and satisfies a certain continuity property. To define the latter, we consider $\Lip(X)$ as a partially ordered vector space with respect to the pointwise ordering. A \emph{monotone net} of functions is a family $(f_\alpha)_{\alpha\in I}$ in $\Lip(X)$ indexed by a directed set $I$, such that $f_\alpha\le f_\beta$ if $\alpha\le\beta$. If the supremum $\sup_\alpha f_\alpha$ exists in $\Lip(X)$, we say that this supremum is \emph{pointwise} if $(\sup_\alpha f_\alpha)(x) = \sup_\alpha f_\alpha(x)$ for every $x\in X$. For example on $X = [0,1]$, the sequence of functions
\begin{equation}
\label{non-pointwise_sup}
	f_n(x) := \min(nx,1)
\end{equation}
with Lipschitz constant $n\in\N$ is a monotone sequence in $\Lip([0,1])$ whose supremum is the constant function $1$, but this supremum is not pointwise, since $(\sup_n f_n)(0) = 1$, while $\sup_n f_n(0) = 0$.

The following representation theorem is similar to~\cite[Theorem~2.4.12]{edgar} and essentially a special case of~\cite[Theorem~436H]{fremlin}.

\begin{thm}\label{fl-complete}
 Let $X\in\cat{Met}$. Mapping every probability measure to its expectation value functional, $p\mapsto \E_p$, establishes a bijective correspondence between probability measures on $X$ with finite first moment, and linear functionals $\phi:\Lip(X)\to\R$ with the following properties:
 \begin{itemize}
  \item Positivity: $f\ge 0$ implies $\phi(f)\ge 0$;
  \item $\tau$-smoothness: if $(f_\alpha)_{\alpha\in I}$ is a monotone net in $\Lip(X)$ with pointwise supremum $\sup_\alpha f_\alpha\in\Lip(X)$, then
\begin{equation}
\label{tau_smooth}
	\phi\left(\sup_\alpha f_\alpha\right) = \sup_\alpha \phi(f_\alpha).
\end{equation}
  \item Normalization: $\phi(1)=1$.
 \end{itemize}
\end{thm}

The concept of $\tau$-smoothness is similar to \emph{Scott continuity} in the context of domain theory and to \emph{normality} in the context of von Neumann algebras, but the important difference is that the preservation of suprema only applies to pointwise suprema: the pointwiseness expresses exactly the condition that integration against delta measures must preserve the supremum. For example, integrating \eqref{non-pointwise_sup} against $\delta_0$ does not preserve the supremum.

\begin{proof}
The fact that the map $p\mapsto \E_p$ is surjective onto functionals satisfying the above conditions is an instance of~\cite[Theorem~436H]{fremlin}. It remains to be shown that the representing measure $p$ is unique. If $\E_p = \E_q$, then by~\cite[Proposition~416E]{fremlin}, it is enough to show that $p(U) = q(U)$ for every open $U \subseteq X$. But now the sequence $(f_n)$ of Lipschitz functions
\[
	f_n(x) := \min(1, n\cdot d(x,X\setminus U) )
\]
monotonically converges pointwise to the indicator function of $U$. Together with Lebesgue's monotone convergence theorem, the equality $\E_p = \E_q$ therefore implies $p(U) = q(U)$, as was to be shown.
\end{proof}

We collect another property for future use, which relies crucially on the non-negativity of a measure:

\begin{lemma}
Let $p\in PX$ and $f:X\to Y$ continuous such that the pushforward measure $f_*p$ is supported on some subset $Y'\subseteq Y$. Then $p$ is supported on $f^{-1}(Y')$.
\label{pullsupport}
\end{lemma}

\begin{proof}
For $x\in X\setminus f^{-1}(Y')$, by assumption there is a neighborhood $U\ni f(x)$ to which $f_*p$ assigns zero measure. Therefore $(f_*p)(U) = p(f^{-1}(U)) = 0$, and $f^{-1}(U)$ is a neighborhood of $x$.
\end{proof}

\subsection{Construction of the Wasserstein space}
\label{conwsp}

\begin{deph}
 Let $X\in\cat{Met}$. The \emph{Wasserstein space} $PX$ is the set of probability measures on $X$ with finite first moment, with metric given by the \emph{Wasserstein distance}, or \emph{Kantorovich-Rubinstein distance}, or \emph{earth mover's distance}\footnote{For the different names, see the bibliographical notes at the end of Chapter 6 in \cite{villani}.}:
\begin{equation}
\label{W1def}
d_{PX}(p,q):= \inf_{r\in\Gamma(p,q)} \int_{X\times X} d_X(x,y) \, dr(x,y)
\end{equation}
where $\Gamma(p,q)$ is the set of probability measures on $X\times X$ with marginals $p$ and $q$, respectively.
\end{deph}

In terms of duality, one can also characterize the Wasserstein metric as
\begin{equation}\label{weak}
d_{PX}(p,q) = \sup_{f: X\to \R} \left| \int f(x) \, d(p-q)(x) \right| = \sup_{f:X\to\R} \left(\E_p[f] - \E_q[f]\right),
\end{equation}
where the sup is taken over all short maps \cites{villani,opttrans}, which we think of as well-behaved random variables. This duality formula provides one way to see that $d_{PX}$ is in fact a metric.

A simple special case of the Wasserstein distance is:

\begin{lemma}
Let $\delta(x_0)$ be the Dirac measure at some $x_0\in X$.\footnote{The notation $\delta(x_0)$ suggests a map $\delta:X\to PX$. This is indeed the case, as we will see in \ref{msk}.} Then
\begin{equation}
	d(\delta(x_0),p) = \int d(x_0,x)\, dp(x).
\label{Wdeltapeq}
\end{equation}

\label{Wdeltap}
\end{lemma}

\begin{proof}
The only joint that has $\delta(x_0)$ as its first marginal and $p$ as its second marginal is the product measure $\delta(x_0)\otimes p$. Therefore,
\begin{align*}
 d(\delta(x_0),p) &= \int_{X\times X} d(y,x) \, d(\delta(x_0)\otimes p)(x,y)\\
 &= \int_{X\times X} d(y,x) \, d(\delta(x_0))(y)\,d p(x) \\
 &= \int_X d(x_0,x)\, dp(x) .\qedhere
\end{align*}
\end{proof}

\begin{thm}[{\cite[Theorem 1.8]{hitch}}]
\label{edgarthm}
 Let $X\in\cat{CMet}$. Then $PX$ is also a complete metric space.
\end{thm}

Moreover, if $X$ is separable (resp.~compact), then $PX$ is also separable (resp.~compact), as proven for example in \cite[Theorem~6.18]{villani}. 

\begin{lemma}
If $f:X\to Y$ is an isometric embedding, then so is $Pf:PX\to PY$.
\label{Pisoemb}
\end{lemma}

\begin{proof}
This follows from the duality formula~\eqref{weak} and the fact that, for $X\subseteq Y$, every $1$-Lipschitz function $g:X\to\R$ can be extended to $Y$, for example via
\[
y\longmapsto \sup_{x\in X} \: \left( g(x) - d(x,y)\right). \qedhere
\]
\end{proof}

We would like the construction $X\mapsto PX$ to be functorial in $X$, and this indeed turns out to be the case. For $f:X\to Y$, we define $Pf:PX\to PY$ to be the map which takes every measure to its pushforward $f_*p\in PY$. In the dual picture, in terms of functionals, $f_*p$ is characterized by the substitution formula: for every $g:Y\to\R$,
\begin{equation}\label{pushforward}
\E_{f_*p}(g) = \int_Y g(y) \, d(f_*p)(y) = \int_X g(f(x))\, dp(x) = \E_p(g\circ f).
\end{equation}
While preservation of composition and identities are clear, there are still two small things to check in order to establish functoriality:

\begin{lemma}\label{Pfunctor}
Let $f:X\to Y$ be short, and $p\in PX$. Then,
\begin{enumerate}
\item $f_*p$ has finite first moment as well;
\item $f_*:PX\to PY$ is short.
\end{enumerate}
\end{lemma}

\begin{proof}
\begin{enumerate}
\item 
Let $g:Y\to\R$ be a Lipschitz map, we have $\E_{f_*p}(g) = \E_p(g\circ f) < \infty$ by~\eqref{pushforward} and by the assumption that $p$ has finite first moment.
\item 
\begin{align*}
d_{PY} \big( f_*p, f_*q \big) &= \sup_{g:Y\to \R} \left( \E_{f_*p}(g) - E_{f_*q}(g)\right) = \sup_{g:Y\to\R} \left( \E_p(g\circ f) - E_q(g\circ f)\right) \\
&\le \sup_{h:X\to\R} \left( \E_p(h) - \E_q(h)\right) = d_{PX} \big( p, q \big).\qedhere
\end{align*}
\end{enumerate}
\end{proof}

Thus we have a functor $P:\cat{Met}\to\cat{Met}$. By Theorem~\ref{edgarthm}, $P$ restricts to an endofunctor of $\cat{CMet}$, which we also denote by $P$. This is the functor we study in this paper. We call it the \emph{Kantorovich functor}, in accordance with \cite{breugel}. 

\section{The Wasserstein space as a colimit}
\label{seccolimit}

Thanks to the metric structure, it turns out that for $X\in\cat{CMet}$, the Wasserstein space $PX$ also arises as the colimit of a diagram involving certain powers of $X$. The intuition behind this colimit is very operational and formalizes the idea that a probability measure is an idealized version of a finite ensemble of elements of $X$, sampled randomly via repeated trials. In the next section, we will exploit this colimit characterization in order to equip $P$ with a monad structure.

\subsection{Power functors}
\label{powerfunsec}

For $X\in\cat{Met}$ and $n\in\N$, let $X^n$ be the metric space whose underlying set is the cartesian power, as in the case of $X^{\otimes n}$, but whose distances are renormalized,
        \begin{equation}\label{pfmetric}
        d_{X^n} \big( (x_1,\ldots,x_n), (y_1,\ldots,y_n) \big) := \frac{ d_X(x_1,y_1) + \ldots + d_X(x_n,y_n) }{n}.
        \end{equation}
One way to motivate this renormalization is that the diagonal map $X\to X^{\otimes n}$ is not short\footnote{This is related to the fact that the symmetric monoidal category $(\cat{Met},\otimes)$ is semicartesian, but not cartesian.}, while the diagonal map $X\to X^n$ is an isometric embedding which we call the \emph{$n$-copy embedding}. Another motivation is given in Appendix~\ref{operad}, where we show that $\cat{Met}$ is a pseudoalgebra of the simplex operad in such a way that the power $X^n$ is the uniform $n$-ary ``convex combination'' of $X$ with itself.

Let $X_n$ be the quotient of $X^n$ under the equivalence relation $(x_1,\ldots,x_n)\sim (x_{\sigma(1)},\ldots,x_{\sigma(n)})$ for any permutation $\sigma\in S_n$. The elements of $X_n$ are therefore multisets $\{x_1,\ldots,x_n\}$. The quotient metric is explicitly given by
    \begin{equation}\label{discopt}
    d_{X_n} \big( \{x_1\dots x_n\}, \{y_1\dots y_n\} \big) := \min_{\sigma\in S_n} \frac{1}{n} \sum_{i=1}^n d_X(x_i,y_{\sigma(i)}),
    \end{equation}
since this is exactly the minimal distance between the two relevant fibers of the quotient map $q_n : X^n\to X_n$, and these distances already satisfy the triangle inequality. Due to this formula, the composite $X\to X^n \to X_n$ is also an isometric embedding, which we call the \emph{symmetrized $n$-copy embedding} $\delta_n:X\to X_n$. It is clear that the assignments $X\mapsto X^n$ and $X\mapsto X_n$ are functorial in $X\in\cat{Met}$, so that we have functors $(-)^n:\cat{Met}\to\cat{Met}$ and $(-)_n:\cat{Met}\to\cat{Met}$. The quotient map is a natural transformation $q_n:(-)^n\Rightarrow(-)_n$.

There is a simple alternative way to write the metric~\eqref{discopt} that connects to the Wasserstein distance~\eqref{W1def}:

\begin{lemma}
\label{discKant}
\begin{equation}
	\label{discopt2}
	d_{X_n}\left( \{x_i\},\{y_i\}\right) = \min_A \frac{1}{n} \sum_{i,j} A_{ij}\, d(x_i, y_j),
\end{equation}
where $A$ ranges over all bistochastic matrices\footnote{We recall that a bistochastic matrix is a square matrix with non-negative entries, whose row and columns all sum to one.}.
\end{lemma}

\begin{proof}
The right-hand side of~\eqref{discKant} is upper bounded by~\eqref{discopt} since every permutation matrix is bistochastic. Conversely, every bistochastic matrix is a convex combination of permutation matrices: according to the Birkhoff-von Neumann theorem, the bistochastic matrices of fixed dimension form a polytope whose vertices are precisely the permutation matrices. Therefore the linear optimization of~\eqref{discopt2} attains the minimum on a permutation matrix.
\end{proof}

It is not hard to see that if $X$ is complete, then so is every $X^S$. Since every $X_n$ is a coequalizer of $X^n$ via~\eqref{discopt} by the action of a finite group, it follows that $X_n$ is complete as well: if $(x_k)_{k\in\N}$ is a Cauchy sequence in $X_n$, then we can assume without loss of generality that $d(x_k,x_{k+1}) \leq 2^{-k}$ after passing to a subsequence. Then we can lift every $x_k \in X_n$ to $\hat{x}_k \in X^n$, in such a way that $d(\hat{x}_k,\hat{x}_{k+1}) \leq 2^{-k}$ as well, which implies that $(\hat{x}_k)_{k\in\N}$ is also Cauchy and therefore convergent. It follows that $\lim_{k\to\infty} x_k = q\left( \lim_{k\to\infty} \hat{x}_k \right)$, so that $X_n$ is complete.

\begin{lemma}
If $f:X\to Y$ is an isometric embedding, then so are $f^n : X^n\to Y^n$ and $f_n : X_n\to Y_n$.
\label{powerisoemb}
\end{lemma}


Categorically, it is more natural to consider the powers $X^S$ for nonempty finite sets $S$, where $X^S$ is the metric space whose elements are functions $x_{(-)}:S\to X$ equipped with the rescaled $\ell^1$-metric,
\[
    d_{X^S} \big( x_{(-)}, y_{(-)} \big) := \frac{1}{|S|} \sum_{s\in S} d_X(x_s,y_s).
\]
The idea is that the points of $X^S$ are finite samples indexed by a set of observations $S$, and a function $x_{(-)}:S\to X$ assigns to every observation $s$ its outcome $x_s$. Then it is natural to define the distance between two finite sets of observations as the average distance between the outcomes.

It is clear that $X^S$ is functorial in $X$, but how about functoriality in $S$? Without the rescaling, we would have functoriality $X^T\to X^S$ for arbitrary injective $S\to T$, corresponding to semicartesianness of $(\cat{Met},\otimes)$. But due to the rescaling by $\tfrac{1}{|S|}$, the functoriality now is quite different:

\begin{lemma}
\label{powerfunc}
Whenever $\phi:S\to T$ has \emph{fibers of uniform cardinality}, i.e.~when the cardinality of $\phi^{-1}(t)$ does not depend on $t\in T$, we have an isometric embedding $-\circ\phi:X^T\to X^S$.
\end{lemma}

We also denote this map $-\circ\phi$ by $X^\phi$.

\begin{proof}
 Let $x_{(-)},y_{(-)}\in X^T$. Then:
\begin{align*}
d_{X^S} \big( X^\phi(x_{(-)}), X^\phi(y_{(-)}) \big) &= d_{X^S} \big( (x_{\phi(-)}), (y_{\phi(-)}) \big) \\
&= \frac{1}{|S|} \sum_{s\in S} d_X(x_{\phi(s)},y_{\phi(s)}) \\
&= \frac{1}{|S|} \sum_{t\in T} |\phi^{-1}(t)| \, d_X(x_{t},y_{t}) \\
&= \frac{1}{|S|} \frac{|S|}{|T|} \sum_{t\in T} d_X(x_{t},y_{t}) \\
&= d_{X^T} \big( x_{(-)}, y_{(-)}\big). \qedhere
\end{align*}
\end{proof}

\begin{deph}
Let $\cat{FinUnif}$ be the monoidal category where:
 \begin{itemize}
  \item Objects are nonempty finite sets;
  \item Morphisms are functions $\phi:S\to T$ with fibers of uniform cardinality,
    \begin{equation}
    |\phi^{-1}(t)| = |S|/|T| \quad \forall t\in T.
    \end{equation}
  \item The monoidal structure is given by the cartesian product\footnote{This is not the categorical product. In fact, $\cat{FinUnif}$ does not have any nontrivial products, but it is semicartesian monoidal.}.
 \end{itemize}   
\end{deph}

In particular, $\cat{FinUnif}$ contains all bijections between nonempty finite sets, and all its morphisms are surjective maps. If we think of every finite set as carrying the uniform probability measure, then $\cat{FinUnif}$ is precisely the subcategory of $\cat{FinSet}$ which contains the measure-preserving maps.

In the following, we either use the powers $X^S$ for finite sets $S\in\cat{FinUnif}$, or equivalently the $X^n$. In the latter case, we take the $n$ to be the objects of a skeleton of $\cat{FinUnif}$ indexed by positive natural numbers $n$. By equivalence of categories, the choice between the two approaches is only a matter of notation. 

We write $X^{(-)}:\cat{FinUnif}^\op\to \cat{CMet}$ for the \emph{power functor} corresponding to Lemma~\ref{powerfunc}.

\begin{deph}
 Let $\cat{N}$ be the monoidal poset of positive natural numbers $\N\setminus\{0\}$ ordered by reverse divisibility, so that a unique morphism $n\to m$ exists if and only if $m|n$, and monoidal structure given by multiplication.
\end{deph}

$\cat{N}$ is the posetification of $\cat{FinUnif}$, in the sense that the canonical functor $|-|:\cat{FinUnif}\to \cat{N}$ which maps every $S$ to its cardinality is the initial functor from $\cat{FinUnif}$ to a poset. Since $|S\times T| = |S| \cdot |T|$, this functor is strict monoidal.

In analogy with the power functor $X^{(-)}:\cat{FinUnif}^\op\to\cat{CMet}$, we can also consider the \emph{symmetrized power functor} $X_{(-)}:\cat{N}^\op\to\cat{CMet}$ which takes $n\in\cat{N}$ to $X_n$, and the unique morphism $m\to m n$, or $m|mn$, is mapped to to the embedding $X_{m|mn}:X_m\to X_{m n}$ given by $n$-fold repetition on multisets,
    \begin{equation}
    \{x_1,\dots, x_m\} \longmapsto \{x_1,\dots x_m, \:\dots, \: x_1,\dots, x_m\}.
    \end{equation}
which is clearly natural in $X$. One can also consider this map as the bottom arrow of a diagram of the form
\begin{equation}\begin{tikzcd}[column sep=large]
\label{imnquot}
	X^T \ar{r}{X^\phi} \ar{d} & X^S \ar{d} \\
	X_{|T|} \ar{r}{X_{ |T| \, | \, |S|}} & X_{|S|}
\end{tikzcd}\end{equation}
The bottom arrow is determined uniquely by the universal property of the quotient map on the left.

\begin{lemma}
$X_{m|mn}:X_m\to X_{m n}$ is an isometric embedding.
\end{lemma}

\begin{proof}
Let $\{x_i\},\, \{y_i\} \in X_m$. Then using Lemma~\ref{discKant}, we can write
\[
d_{X_{mn}}\left(X_{m|mn}(\{x_i\}), X_{m|mn}(\{y_i\})\right) = \frac{1}{mn} \min_A \sum_{i,j,\alpha,\beta} A_{(i,\alpha),(j,\beta)}\, d_X(x_i, y_j),
\]
where $A$ ranges over all bistochastic matrices of size $mn\times mn$ with rows and columns indexed by pairs $(i,\alpha)$ with $i=1,\ldots,m$ and $\alpha=1,\ldots,n$. Similarly,
\[
d_{X_m}\left(\{x_i\},\{y_i\}\right) = \frac{1}{m} \min_B \sum_{i,j} B_{ij}\, d_X(x_i, y_j).
\]
For given $B$, one can achieve the same value of the first optimization by putting~$A_{ij}^{\alpha\beta} := \tfrac{1}{n}B_{ij}$ for all values of the indices. Conversely, we can put $B_{ij} := \tfrac{1}{n} \sum_{\alpha,\beta} A_{(i,\alpha),(j,\beta)}$ in order to achieve the same value in the other direction.
\end{proof}

Thus the symmetrized power functor $X_{(-)}:\cat{N}^\op\to \cat{CMet}$ lands in the subcategory of complete metric spaces and isometric embeddings.

Again we have a quotient map $q_S:X^S\to X_{|S|}$ given by ``forgetting the labeling'' of particular outcomes and only remembering the multiset of values of the given function $x_{(-)}:S\to X$,
\begin{equation}\label{defq}
q_S\left( x_{(-)} \right) = \{x_s \,:\, s\in S\} \in X_{|S|}.
\end{equation}
It is the universal morphism which coequalizes all automorphisms of $X^S$ of the form $X^\sigma$, where $\sigma$ ranges over all bijections $\sigma:S\to S$.

In this way, we obtain a natural transformation $q: X^{(-)}\Rightarrow X_{|-|}$ of functors $\cat{FinUnif}^\op\to\cat{CMet}$.

\begin{lemma}
Via the map $q$ appearing in the formula~\eqref{defq}, the functor $X_{(-)}:\cat{N}^\op\to\cat{Met}$ is the left Kan extension of $X^{(-)}:\cat{FinUnif}^\op\to\cat{Met}$ along $|-|^\op$. Likewise for $\cat{CMet}$ in place of $\cat{Met}$.
\label{kanext}
\end{lemma}

\begin{proof}
Again because $\cat{CMet}\subseteq\cat{Met}$ is reflective, it is enough to prove the claim for $\cat{Met}$.
There it follows from the universal property of the quotient map $q$. In more detail, we have the diagram
\[\begin{tikzcd}
\cat{FinUnif}^\op \ar{rr}{X^{(-)}} \ar[swap]{dr}{|-|} 
 \ar[bend right=8,""{name=XUP,pos=0.46},phantom]{rr} \ar[bend right=37,""{name=XDOWN,pos=0.532},phantom]{rr}
  \ar[Rightarrow,from=XUP,to=XDOWN,"q"]
 & & \cat{Met} \\
& \cat{N}^\op \ar[swap]{ur}{X_{(-)}}
\end{tikzcd}\]
Consider now another functor $K$ and a natural transformation $\alpha$ as in
\[\begin{tikzcd}
\cat{FinUnif}^\op \ar{rr}{X^{(-)}} \ar[swap]{dr}{|-|} 
 \ar[bend right=8,""{name=XUP,pos=0.46},phantom]{rr} \ar[bend right=37,""{name=XDOWN,pos=0.532},phantom]{rr}
  \ar[Rightarrow,from=XUP,to=XDOWN,"\alpha"]
 & & \cat{Met} \\
& \cat{N}^\op \ar[swap]{ur}{K}
\end{tikzcd}\]
Unraveling the definition, this means that for each $S\in\cat{FinUnif}$ we have a map
\[
\alpha_S : X^S \to K(|S|),
\]
and we need to find a factorization
\begin{equation}\label{defu}\begin{tikzcd}
X^S \ar[swap]{dr}{\alpha_{S}} \ar{r}{q} & X_{|S|} \uni{d}{u_{|S|}} \\
& K(|S|)
\end{tikzcd}\end{equation}
for some $u:X_{(-)}\Rightarrow K$. By naturality of $\alpha$ with respect to automorphisms $\sigma:S\to S$, we know that $\alpha_S$ is invariant under precomposing by $X^\sigma$. Therefore it factors uniquely through $q$ and this defines $u_{|S|}$, which is enough since $|-|$ is (essentially) bijective on objects. It remains to prove naturality of $u$, which means that for all $m,n\in\cat{N}$, the diagram
\[\begin{tikzcd}
X_m \ar{d}{u_m} \ar{rr}{X_{m|mn}} & & X_{mn} \ar{d}{u_{mn}} \\
K(m) \ar{rr}{K(m|mn)} & & K(mn)
\end{tikzcd}\]
commutes. This follows from the fact that $|-|:\cat{FinUnif}\to\cat{N}$ is also full, so that the morphism $X_{m|mn}$ is the image of some morphism in $\cat{FinUnif}$, together with naturality of $\alpha$ and the definition~\eqref{defu}.
\end{proof}

Equivalently, we could have obtained this result from the usual coend formula for left Kan extensions: although general copowers do not exist in $\cat{Met}$ or $\cat{CMet}$, since we require all distances to be finite, they do exist in this particular case since $\cat{N}^\op$ is a poset, so that the coend formula only requires the existence of the trivial copowers by the empty set and by singleton sets.

In conclusion, we also have endofunctors $(-)^S:\cat{CMet}\to\cat{CMet}$ and $(-)_n:\cat{CMet}\to\cat{CMet}$.

\subsection{Empirical distribution maps}
\label{empdist}

\begin{deph}
 Let $X\in\cat{Met}$. For $S\in\cat{FinUnif}$, the \emph{empirical distribution map} is the map $\iota^S:X^S\to PX$ which assigns to each $S$-indexed family $x_{(-)}\in X^S$ the uniform probability measure,
  \begin{equation}
   \iota^S(x_{(-)}) := \dfrac{1}{|S|}\sum_{s\in S} \delta(x_s).
  \end{equation}
\end{deph}

This map is clearly permutation-invariant, so it uniquely determines a map on symmetrized powers as well:

\begin{deph}
For $n\in\cat{N}$, the \emph{symmetric empirical distribution map} is the map $\iota_n:X_n\to PX$ given by assigning to each multiset $\{x_1,\dots, x_n\} \in X_n$ the corresponding uniform probability measure,
\begin{equation}
\iota_n (\{x_1\dots x_n\}) := \frac{\delta(x_1) + \dots + \delta(x_n)}{n}\;.
\end{equation}
\end{deph}

The empirical distribution carries less information than the original sequence. The information lost is precisely the ordering, as the following proposition shows.

\begin{prop}\label{isodelta}
 $\iota_n:X_n\to PX$ is an isometric embedding for each $X$ and $n$.
\end{prop}

\begin{proof}
For $\{x_i\},\{y_i\}\in X_n$, consider the finite set 
$$
N_{xy} := \{1_x,\ldots,n_x\}\amalg\{1_y,\ldots,n_y\},
$$
and equip it with the pseudometric which makes the canonical map $N_{xy}\to X$ into an isometric embedding, which means in particular that $d(i_x,j_y) = d_X(x_i,y_j)$. In the commutative square
\[
\begin{tikzcd}
N_{xy,n} \ar{r}{\iota_n} \ar{d} & PN_{xy} \ar{d} \\
X_n \ar{r}{\iota_n} & PX
\end{tikzcd}
\]
both vertical arrows are isometric embeddings by Lemmas~\ref{powerisoemb} and~\ref{Pisoemb}, where both $P$ and the lemmas generalize to pseudometric spaces in the obvious way. It is therefore enough to prove that in $PN_{xy}$, the distance between the uniform distribution on the points $\{1_x,\ldots,n_x\}$ and $\{1_y,\ldots,n_y\}$ is equal to the distance between these two sets as elements of $N_{xy,n}$. This is indeed the case, since the latter distance is given by~\eqref{discopt2},
\[
	d(\{i_x\},\{j_y\}) = \min_A \frac{1}{n}\sum_{ij} A_{ij}\, d(x_i, y_j),
\]
where $A$ ranges over all bistochastic matrices, which means exactly that $\tfrac{1}{n}A$ ranges over all couplings between the two uniform marginals as in the definition of the Wasserstein distance~\eqref{W1def}.
\end{proof}
 
It is clear that $\iota^S$ is natural in $X$, so we consider it as a transformation $\iota^S:(-)^S\Rightarrow P$ between the power functor at $S$ and the Kantorovich functor. Similarly, $\iota_n:(-)_n\Rightarrow P$.

\begin{lemma}\label{imn}
 Let $n,m\in\cat{N}$, and $X\in\cat{CMet}$. Then the following diagram commutes:
\begin{equation}\begin{tikzcd}
X_m \ar[swap]{dr}{\iota_m} \ar{rr}{X_{m|mn}} && X_{m n} \ar{dl}{\iota_{m n}} \\
& PX
\end{tikzcd}\end{equation}
\end{lemma}

\begin{proof}
For $\{x_1,\dots, x_m\}\in X_m$,
\begin{align*}
\iota_{mn} \circ X_{m|mn} &(\{x_1\dots   x_m\}) =  \iota_{m n}  (\{x_1\dots x_m,\:\dots\:, x_1,\dots, x_m\}) \\[3pt]
&= \frac{\delta(x_1) + \dots + \delta(x_m) + \dots + \delta(x_1) + \dots + \delta(x_m)}{mn} \\
&= \frac{\delta(x_1) +\dots + \delta(x_m)}{m} 
= \iota_m (x_1\dots x_m). \qedhere
\end{align*}
\end{proof}

Therefore the symmetrized empirical distribution map $\iota_n$ is natural in $n$. It follows that the empirical distribution map $\iota^S$ is natural in $S$.

\subsection{Universal property}
\label{uniprop}

\begin{deph}
 Let $X$ be a complete metric space, and consider the symmetrized empirical distribution maps $\iota_n:X_n\to PX$, which are embeddings for each $n\in \cat{N}$. We write $I(X)$ for the union of their images,
\begin{equation}
I(X) := \bigcup_{n\in \cat{N}} \iota_n(X_n) \subseteq PX\;.
\end{equation}
\end{deph}

\begin{lemma}
$I(X)$ is the colimit of the functor $X_{(-)}:\cat{N}^\op\to\cat{Met}$, and also of the functor $X^{(-)}:\cat{FinUnif}^\op\to\cat{Met}$, with the $\iota_n$ and the $\iota^S$ forming the universal cocones.
\label{colimitmet}
\end{lemma}

\begin{proof}
By Lemma~\ref{kanext} and composition of Kan extensions, it is enough to prove this for $X_{(-)}$. So let the $\{g_n : X_n\to Y\}$ form a cocone, i.e.~a family of short maps such that $g_m = g_{mn} X_{m|mn}$. Since the $\iota_n:X_n\to I(X)$ are jointly epic by definition of $I(X)$, there can be at most one map $I(X)\to Y$ that is a morphism of cocones, which establishes uniqueness. Concerning existence, every point of $I(X)$ is of the form $\iota_n(\{x_i\})$ for some $n$ and some $\{x_i\}\in X_n$, and we therefore define its image in $Y$ to be $g_n(\{x_i\})$. This is well-defined: if $\iota_n(\{x_i\}) = \iota_m(\{x'_j\})$, then the relative frequencies of all points of $X$ in the multiset $\{x_i\}$ must coincide with those in $\{x'_j\}$. In particular this implies $X_{m|mn}(\{x_i\}) = X_{n|mn}(\{x'_j\})$, which is enough by the assumed naturality of the $\{g_m\}$. Finally, the resulting map is still short since any two points in $I(X)$ come from some common $X_n$, and $\iota_n:X_n\to I(X)$ is an isometric embedding.
\end{proof}

$I(X)$ is not complete unless $|X|\leq 1$. The following result is essentially proven in \cite[Proposition~1.9]{hitch} by reduction to the separable case treated in \cite{villani}. We give here an alternative proof that works without mentioning separability.

\begin{thm}\label{dense}
Let $X$ be a metric space. Then $I(X)$ is dense in $PX$.
\end{thm}

We prove this in several steps, starting with the compact case.

\begin{lemma}
 Let $X$ be a compact metric space. Then $I(X)$ is dense in $PX$.
\end{lemma}

\begin{proof}
First of all, let $p\in PX$ be finitely supported, but not necessarily with rational coefficients. The measure $p$ is of the form $p = p_1\,\delta_{x_1} + \dots p_n\,\delta_{x_n}$ for some fixed finite $n$ and positive coefficients $p_i$. For every $\e>0$ and $i=1,\dots,n-1$, we can find rational numbers $q_i$ such that $q_i\le p_i$ and $p_i-q_i<\e$. Define also $q_n:=1-\sum_i q_i$, which is rational as well. The measure $q:= q_1\,\delta_{x_1} + \dots q_n\,\delta_{x_n}$ is then an element of $I(X)$. Moreover, denoting by $D$ the (finite) diameter of $X$, 
\begin{align*}
d(p,q) &\le \sum_{i=1}^{n-1} (p_i - q_i) \cdot d(x_i,x_n) \\
 &\le \sum_{i=1}^{n-1} (p_i - q_i) \cdot D \\
 &\le (n-1)\cdot \e \cdot D.
\end{align*}
Since this can be made arbitrarily small, $I(X)$ is dense in the space of (arbitrary) finitely supported probability measures.
It remains to be shown that finitely supported probability measures are dense in $PX$.

For given $\e > 0$, the open sets of diameter at most $\e$ cover $X$. By compactness, already finitely many of these, say $U_1,\ldots,U_n$, cover $X$. Consider the Boolean algebra generated by the $U_i$, its atoms are measurable sets $A_1,\ldots,A_k$ of diameter at most $\e$ which partition $X$.

$\{A_i\}$ is then a finite family of measurable subsets, mutually disjoint, which cover $X$. Choosing arbitrary $y_i\in A_i$, we have $d(x_i,y_i)<\e$ for every $x_i\in A_i$. For given $p\in PX$, the probability measure
\begin{equation}
p_{\e} := \sum_{i=1}^k p(A_i) \, \delta(y_i)\;.
\end{equation}
is finitely supported. To see that it is close to $p$, we choose a convenient joint,
\begin{equation}
m := \sum_{i=1}^k p|_{A_i} \otimes \delta(y_i),
\end{equation}
where $p|_{A_i}$ is the measure with $p|_{A_i}(B) := p(B\cap A_i)$. Therefore
\begin{align*}
d_{PX}(p,p_{\e}) &\le \int_{X\times X} d_X(x,y) \,dm(x,y) = \sum_i \int_{A_i\times X} d_X(x,y) \,dp(x) \,\delta(y_i)(y)\,dy \\
&= \sum_i \int_{A_i} d_X(x,y_i) \,dp(x) \le \sum_i \int_{A_i} \e \,dp(x) =\e \sum_i p(A_i) = \e,
\end{align*}
as was to be shown.
\end{proof}

Before getting to the general case, we record another useful fact.

\begin{lemma}
\label{dcontract}
Let $p,q_1,q_2\in PX$ and $\lambda\in [0,1]$. Then
\begin{equation}
\label{dcontracteq}
	d_{PX}\big(\lambda q_1 + (1-\lambda)p,\,\lambda q_2 + (1-\lambda)p\big) = \lambda\, d_{PX}(q_1,q_2).
\end{equation}
\end{lemma}

This follows immediately from the duality~\eqref{weak}, but it is instructive to derive the inequality `$\le$' directly by using the fact that any coupling $r\in\Gamma(q_1,q_2)$ gives a coupling
\[
	\lambda r + (1-\lambda) (P\Delta)(p) \quad\in\quad\Gamma\big(\lambda q_1 + (1-\lambda)p,\, \lambda q_2 + (1-\lambda)p\big) 
\]
where $\Delta : X\to X\times X$ is the diagonal embedding, and the second term does not contribute to the expected distance as it is supported on the diagonal.

\begin{lemma}
 Let $X$ be a metric space. Then the set of compactly supported probability measures is dense in $PX$. 
\end{lemma}

\begin{proof}
We first show that boundedly supported measures are dense in $PX$ by finite first moment, and then that compactly supported measures are dense in boundedly supported measures by tightness.

For the first part, let $p\in PX$ and $x_0\in X$ be given. Let $B(x_0,\rho)$ be the closed ball of radius $\rho>0$ around $x_0$. We would like to approximate $p$ by the boundedly supported measure $p|_{B(x_0,\rho)}$, but this is not normalized. The most convenient way to fix this is to use
\[
 p' := p|_{B(x_0,\rho)} + p(X{\setminus} B(x_0,\rho)) \,\delta(x_0)
\]
By decomposing
 \begin{equation}\label{dec-pk}
  p = p|_{B(x_0,\rho)} +  p|_{X{\setminus} B(x_0,\rho)}
 \end{equation}
we can compute
\begin{align*}
  d_{PX}(p,p') & \stackrel{\eqref{dcontracteq}}{=} p\left(X{\setminus}B(x_0,\rho)\right)\, d_{PX}\left(\delta(x_0), \frac{p|_{X{\setminus} B(x_0,\rho)}}{p(X\setminus B(x_0,\rho))}\right) \stackrel{\eqref{Wdeltapeq}}{=} \int_{X\setminus B(x_0,\rho)} d(x,x_0)\, dp(x) \\[4pt] 
 & = \int_{X} d(x,x_0)\,dp(x) - \int_{B(x_0,\rho)} d(x,x_0)\,dp(x).
\end{align*}
The second term on the right-hand side is the expectation value of the function
 \begin{equation}
  f_\rho(x) := \begin{cases} d(x,x_0) & \textrm{if } d(x,x_0) \le \rho, \\ 0 & \textrm{otherwise}. \end{cases}
 \end{equation}
which converges pointwise to $d(-,x_0)$ as $\rho\to\infty$. By monotone convergence, this term converges to the first term, $\int_X d_X(x,x_0)\,dp(x)$, which is finite by the assumption of finite first moment. Hence $d_{PX}(p,p')\to 0$ as $\rho\to\infty$, and the approximating measures $p'$ are boundedly supported.

For the second part, we can then assume that $\diam(X)<\infty$. Let $p\in PX$. For suitably large compact $K\subseteq X$, we would like to approximate $p$ by the compactly supported measure $p|_{K}$, where $p|_K(A) := p(A\cap K)$, but this is not normalized. The most convenient way to fix this is to choose an arbitrary point $x_0\in K$, and to use
 \begin{equation}
  p' := p|_K + p(X{\setminus} K)\,\delta(x_0),
 \end{equation}
By decomposing
 \begin{equation}\label{dec-pk2}
  p = p|_K +  p|_{X\backslash K},
 \end{equation}
 we can compute
\[
	d_{PX}(p,p') \stackrel{\eqref{dcontracteq}}{=} p(X{\setminus} K)\: d\left(\frac{p_{X{\setminus} K}}{p(K)},\,\delta(x_0)\right) \stackrel{\eqref{Wdeltapeq}}{=} p(X{\setminus} K)\,\diam(X),
\]
By tightness, this tends to $0$ as $K\rightarrow X$.
\end{proof}

Theorem~\ref{dense} then follows as a corollary. 

We now consider what happens in the reflective subcategory of \emph{complete} metric spaces, $\cat{CMet}\subseteq\cat{Met}$.

\begin{thm}\label{pcolimit}
The space $PX$ is the colimit of the functor $X_{(-)}:\cat{N}^\op\to\cat{CMet}$, and also of the functor $X^{(-)}:\cat{FinUnif}^\op\to\cat{CMet}$.
\label{colimitcmet}
\end{thm}

\begin{proof}
Use Lemma~\ref{colimitmet} together with Theorem~\ref{dense}, and the fact that if $Y$ is a complete metric space with $X\subseteq Y$ dense, then $Y$ is the completion of $X$ with the inclusion as the universal morphism. 
\end{proof}

Since colimits over $\cat{FinUnif}^\op$ or $\cat{N}^\op$ in a category of functors into $\cat{Met}$ or $\cat{CMet}$ are computed pointwise\footnote{Technically, this relies on the fact that such colimits always exists in $\cat{Met}$ and $\cat{CMet}$, per Lemma~\ref{colims_met_lem}.}, this implies that the Wasserstein space construction in the form of the object $P\in[\cat{CMet},\cat{CMet}]$, is the colimit of the power functor construction:

\begin{cor}\label{empdcolimit} The empirical distribution maps form two colimiting cocones in the following way:
\begin{enumerate}
\item Consider the functor $(=)_{(-)}:\cat{N}^\op \to[\cat{CMet},\cat{CMet}]$ mapping $n\in \cat{N}$ to the symmetrized power functor $X\mapsto X_n$.
Then $P\in [\cat{CMet},\cat{CMet}]$ is the colimit of $(=)_{(-)}$, with the colimit cocone given by the symmetrized empirical distribution maps $\iota_n:(-)_n\Rightarrow P$.

\item Consider the functor $(=)^{(-)}:\cat{FinUnif}^\op\to[\cat{CMet},\cat{CMet}]$ mapping $S\in\cat{FinUnif}$ to the power functor $X\mapsto X^S$.
Then $P\in [\cat{CMet},\cat{CMet}]$ is the colimit of $(=)^{(-)}$, with the colimit cocone given by the empirical distribution maps $\iota^S:(-)^S\Rightarrow P$.
\end{enumerate}
\end{cor}

\begin{remark}
Readers concerned with size issues may find it problematic that the endofunctor category $[\cat{CMet},\cat{CMet}]$ is not locally small, so that the above universal properties potentially involve bijections between proper classes (or large sets). One way to see that $[\cat{CMet},\cat{CMet}]$ is not locally small is to borrow the fact that the functor category $[\cat{CMet},\cat{Set}^\op]$ is not small from~\cite{freyd_street_size}, and choose a faithful functor $\cat{Set}^\op\to\cat{CMet}$, such as the composition of the power set functor $\cat{Set}^\op\to\cat{Set}$ with the discrete metric space functor $\cat{Set}\to\cat{CMet}$ which equips every set with the discrete $\{0,1\}$-valued metric. In this way, a large hom-set in $[\cat{CMet},\cat{Set}^\op]$ embeds into a hom-set of $[\cat{CMet},\cat{CMet}]$, which must therefore also be large.

It may be possible to alleviate this problem by uncurrying, using $(=)_{(-)} : \cat{N}^\op\times\cat{CMet}\to\cat{CMet}$ and $(=)^{(-)}:\cat{FinUnif}^\op\times\cat{CMet}\to\cat{CMet}$, as in the theory of graded monads developed in~\cite{ssm}.
\end{remark}

\section{Monad structure}
\label{secmonad}

The main result of this section is that the functor $P$ is part of a monad, with units and compositions defined in a way analogous to the Giry monad \cite{giry}. It was proven in \cite{breugel} that the restriction of $P$ to compact metric spaces carries a monad structure. In the spirit of categorical probability theory, the monad multiplication is given by averaging a measure on measures to a measure, and the unit by assigning to each point its Dirac measure.

An appealing feature of the Kantorovich functor is that its monad structure can be constructed directly from the colimit characterization in terms of the power functors defined in Section~\ref{seccolimit}. This uses the fact that the power functors carry the structure of a monad \emph{graded} by $\cat{FinUnif}^\op$, in the sense of a lax monoidal functor\footnote{An ordinary monad on a category $\cat{C}$ is graded by the terminal category $1$: being a monoid in $[\cat{C},\cat{C}]$, it is equivalently a lax monoidal functor $1\to[\cat{C},\cat{C}]$.} into the endofunctor category $[\cat{CMet},\cat{CMet}]$, and similarly for the symmetrized power functors in terms of $\cat{N}^\op$. This construction of the Kantorovich monad extends the idea that probability measures are formal limits of finite samples to the level of integration.

\subsection{The power functors form a graded monad}
\label{gradedmonad}

As we will see next, the functor $(=)^{(-)}:\cat{FinUnif}^\op\to[\cat{CMet},\cat{CMet}]$ has a canonical strong monoidal structure with respect to the monoidal structure on $\cat{FinUnif}$ given by the cartesian product. We assume the latter to be strict for notational convenience.

Concerning the unit, there is a canonical transformation $\delta:1_{\cat{CMet}}\Rightarrow(=)^1$ with components given by the identity isomorphisms $X\cong X^1$. For the multiplication, we use the currying maps $E^{S,T}:(X^S)^T\cong X^{S\times T}$. It takes a $T$-indexed family of $S$-indexed families $\{\{x_{ij}\}_{i\in S}\}_{j\in T}$ to the $(S\times T)$-indexed family $\{x_{ij}\}_{i\in S,\, j\in T}$. A straightforward computation shows that $E^{S,T}$ preserves distances, since distances add up across all components $i$ and $j$ and get rescaled by $|S|\cdot|T|$ in both cases. It is also clear that $E^{S,T}$ is natural in $X$.

We find it curious that at this stage, both of these structure maps are isomorphisms, resulting in a strong monoidal functor. While the relevant coherence properties are immediate by the universal properties, we state them here for convenient reference.

\begin{thm}\label{monoidalfunctorunsym}
The above structure transformations $\delta$ and $E^{-,-}$ equip the functor $(=)^{(-)}$ with a strong monoidal structure, i.e.~the following diagrams commute for all $X\in\cat{CMet}$:
 \begin{itemize}
  \item The unit triangles
        \begin{equation}\begin{tikzcd}
	  X^S \ar{r}{\delta} \idar{dr} & (X^S)^1 \ar{d}{E^{S,1}}  &  X^S \ar{r}{\delta^S} \idar{dr} & (X^1)^S \ar{d}{E^{1,S}} \\
	  & X^{S\times 1} & & X^{1\times S}
        \end{tikzcd}\end{equation}
  \item The associativity square
        \begin{equation}
	\label{Eupassoc}
	\begin{tikzcd}[row sep=large,column sep=large]
        ((X^R)^S)^T \ar{r}{E^{S,T}} \ar{d}{(E^{R,S})^T} & (X^R)^{S \times T} \ar{d}{E^{R,S \times T}} \\
        (X^{R\times S})^T \ar{r} {E^{R \times S, T}}& X^{R \times S \times T}
        \end{tikzcd}
	\end{equation}
 \end{itemize}
\end{thm}

For the proof, it is enough to verify commutativity at the level of the underlying sets, where these are standard properties of currying which follow from the universal property of exponential objects.

\subsection{The symmetrized power functors form a graded monad}
\label{monpf}

We now move on to consider the analogous structure on the symmetrized power functors $X\mapsto X_n$. By definition, the quotient map $q_n:X^n\to X_n$ is the universal map which coequalizes the action of the symmetric group $S_n$ permuting the factors. In order to analyze the graded monad structure, we need to analyze the power of a power. The four ways of forming a power of a power fit into the square
\begin{equation}
\begin{tikzcd}
	\label{doublequotient}
	(X^m)^n \ar{r}{q_n} \ar{d}{(q_m)^n} & (X^m)_n \ar{d}{(q_m)_n} \\
	(X_m)^n \ar{r}{q_n} & (X_m)_n
\end{tikzcd}
\end{equation}
which commutes by naturality of $q_n$. The left arrow has a universal property as well:

\begin{lemma}
In $\cat{CMet}$, the morphism $(q_m)^n$ is the universal map out of $(X^m)^n$ which coequalizes the action of $(S_m)^{\times n}$ given by acting on each outer factor separately.
\end{lemma}

\begin{proof}
For every space $Y$, the map $Y\otimes q_m : Y\otimes X^m \to Y\otimes X_m$ is also the coequalizer of the $S_m$-action on $X^m$, thanks to Proposition~\ref{tensor_cocont}. Therefore $(q_m)^{\otimes n} : (X^m)^{\otimes n} \to (X_m)^{\otimes n}$ is the coequalizer of the factor-wise action of $(S_m)^{\times n}$ on $(X^m)^{\otimes}$. Finally, the analogous statement for $(q_m)^n : (X^m)^n \to (X_m)^n$ follows by rescaling both metrics by $1/n$, which is an automorphism of the category and therefore preserves colimits.
\end{proof}

It follows that the diagonal morphism is the universal morphism which coequalizes the action of the wreath product group $S_m\wr S_n$, where $S_n$ acts on $(S_m)^{\times n}$ by permutation of the factors. We are not aware of any description for $(q_m)_n$ other than the factorization across $q_n$ by the universal property of the latter.

We now define $E_{m,n}:(X_m)_n\to X_{mn}$ by the universal property of the $S_m\wr S_n$-quotient map $(X^m)^n\to (X_m)_n$ as the unique morphism which makes the diagram
\begin{equation}
\begin{tikzcd}
	\label{Edowndef}
	(X^m)^n \ar{r}{E^{m,n}} \ar{d} & X^{mn} \ar{d}{q_{mn}} \\
	(X_m)_n \ar{r}{E_{m,n}} & X_{mn}
\end{tikzcd}
\end{equation}
commute. Explicitly, $E_{m,n}$ takes a multiset of $n$ multisets of cardinality $m$ and forms the union over the outer layer, resulting in a single multiset of cardinality $mn$. This is a graded version of the multiplication in the commutative monoid monad; in particular, in contrast to the $E^{m,n}$, the $E_{m,n}$ are not isomorphisms (unless $m=1$ or $n=1$). Naturality in $X$ follows directly from the definition. Concerning the unit, we have the composite isomorphism $X\cong X^1 \cong X_1$, which we also denote by $\delta$.

\begin{thm}\label{monoidalfunctorsym}
The above structure transformations $\delta$ and $E_{-,-}$ equip the functor $(=)_{(-)}$ with a lax monoidal structure, meaning that the following diagrams commute for all $X\in\cat{CMet}$:
 \begin{itemize}
  \item The unit triangles
        \begin{equation}\begin{tikzcd}
	  X_m \ar{r}{\delta} \idar{dr} & (X_m)_1 \ar{d}{E_{m,1}}  &  X_m \ar{r}{\delta_m} \idar{dr} & (X_1)_m \ar{d}{E_{1,m}} \\
	  & X_{m\times 1} & & X_{1\times m}
        \end{tikzcd}\end{equation}
  \item The associativity square
        \begin{equation}
	\label{Edownassoc}
	\begin{tikzcd}[row sep=large,column sep=large]
        ((X_\ell)_m)_n \ar{r}{E_{m,n}} \ar{d}{(E_{\ell,m})_n} & (X_\ell)_{m n} \ar{d}{E_{\ell,m n}} \\
        (X_{\ell m})_n \ar{r} {E_{\ell m, n}}& X_{\ell m n}
        \end{tikzcd}
	\end{equation}
 \end{itemize}
\end{thm}

\begin{proof}
We reduce the claim to Theorem~\ref{monoidalfunctorunsym}. Only the associativity square is nontrivial.

By reasoning similarly to~\eqref{doublequotient}, composing the quotient maps results in a unique epimorphism $((X^\ell)^m)^n\to ((X_\ell)_m)_n$. In fact, we get a cube:
\begin{equation}
 \begin{tikzcd}
  & ((X^\ell)^m)_n \ar[near start]{dd}{((q_l)^m)_n} \ar{rr}{(q_m)_n}  && ((X^\ell)_m)_n \ar{dd}{((q_l)_m)_n} \\
  ((X^\ell)^m)^n \ar{dd}{((q_l)^m)^n} \ar{ur}{q_n} \ar[crossing over, near start]{rr}{(q_m)^n} && ((X^\ell)_m)^n \ar{ur}{q_n}  \\
  & ((X_\ell)^m)_n \ar[near end]{rr}{(q_m)_n} && ((X_\ell)_m)_n \\
  ((X_\ell)^m)^n \ar{ur}{q_n} \ar{rr}{(q_m)^n} && ((X_\ell)_m)^n  \ar[leftarrow, near end, crossing over]{uu}{((q_l)_m)^n}\ar{ur}{q_n}
 \end{tikzcd}
\end{equation}
where the top, bottom, right, and left faces commute by naturality of $q_n$, and the front and back faces commute by the naturality of $q_m$. 
Using this, we consider the cube
\begin{equation}
\label{symcube}
\begin{tikzcd}[row sep=large,column sep=large]
        & ((X_\ell)_m)_n \ar{rr}{E_{m,n}} \ar["(E_{\ell,m})_n" near start]{dd} & & (X_\ell)_{m n} \ar{dd}{E_{\ell,m n}} \\
	((X^\ell)^m)^n \ar[crossing over,"E^{m,n}" near start]{rr} \ar{dd}{(E^{\ell,m})^n} \ar{ur} & & (X^\ell)^{mn} \ar{ur} \\
        & (X_{\ell m})_n \ar["E_{\ell m, n}" near end]{rr}  & & X_{\ell m n} \\
	(X^{\ell m})^n \ar{rr}{E^{\ell m,n}} \ar{ur} & & X^{\ell m n} \ar{ur} 
	\ar[leftarrow, swap, crossing over,"E^{\ell, mn}"' near end]{uu}
\end{tikzcd}
\end{equation}
where the unlabeled diagonal arrows are the quotient maps discussed previously. We need to show that the back face commutes. The bottom and right faces commute by~\eqref{Edowndef}. The top face also commutes, thanks to
\[
\begin{tikzcd}
	((X^\ell)^m)^n \ar{rr}{E^{m,n}} \ar{d} && (X^\ell)^{mn} \ar{d} \\
	((X_\ell)^m)^n \ar{rr}{E^{m,n}} \ar{d} && (X_\ell)^{mn} \ar{d} \\
	((X_\ell)_m)_n \ar{rr}{E_{m,n}} && ((X_\ell)_m)_n
\end{tikzcd}
\]
and similarly for the left face. Finally, commutativity of the front face follows from Theorem~\ref{monoidalfunctorunsym}. Therefore, since $((X^\ell)^m)^n\to ((X_\ell)_m)_n$ is epi, this implies that the back face commutes as well.
\end{proof}

We can also consider the $\cat{N}^\op$-graded monad $(=)_{(-)}$ as the universal $\cat{N}^\op$-graded monad that one obtains from the $\cat{FinUnif}^\op$-graded monad $(=)^{(-)}$ by change of grading along $\cat{FinUnif}^\op\to\cat{N}^\op$. In fact, this follows from Lemma~\ref{kanext} and Theorem~\ref{monoidalkanext} in Appendix~\ref{kan}:

\begin{thm}
\label{FinUniftoNmonad}
Let $\cat{MonCat}$ be the bicategory of monoidal categories, lax monoidal functors, and monoidal transformations. Then the lax monoidal functor $(=)_{(-)}:\cat{N}^\op\to[\cat{CMet},\cat{CMet}]$ is the left Kan extension in $\cat{MonCat}$ of $(=)^{(-)}:\cat{FinUnif}^\op\to[\cat{CMet},\cat{CMet}]$ along $\cat{FinUnif}^\op\to\cat{N}^\op$.
\end{thm}

\begin{proof}
By Lemma~\ref{kanext}, this Kan extension works in $\cat{Cat}$, and it is clear that $\cat{FinUnif}^\op\to\cat{N}^\op$ is strong monoidal and essentially surjective. In order to apply Theorem~\ref{monoidalkanext}, it remains to check two things: first, that the transformation $q:(=)^{(-)}\to (=)_{(-)}$ is monoidal, which boils down to the diagram
\[\begin{tikzcd}
	(X^m)^n \ar{r}{E^{m,n}} \ar{d}{q\,\circ\, q} & X^{mn} \ar{d}{q} \\
	(X_m)_n \ar{r}{E_{m,n}} & X_{mn}
\end{tikzcd}\]
which is~\eqref{Edowndef} again. Second, that $q\otimes q$ is an epimorphism in the functor category $[\cat{FinUnif}^\op\times\cat{FinUnif}^\op,[\cat{CMet},\cat{CMet}]]$, which follows from the fact that even every individual double quotient map $(X^m)^n\to (X_m)_n$ is an epimorphism.
\end{proof}

\subsection{The monad structure on the Kantorovich functor}
\label{msk}

Now that we have shifted the graded monad structure from $\cat{FinUnif}^\op$ to $\cat{N}^\op$, we shift it one step further to a lax monoidal functor $1\to[\cat{CMet},\cat{CMet}]$, i.e.~to an ungraded monad on $\cat{CMet}$ whose underlying functor is $P$.

We define the unit and multiplication maps in terms of the power functors and the empirical distribution maps.

\begin{deph}
For $X\in\cat{CMet}$ and $n\in N$, The \emph{Dirac delta embedding} is the composite
\[\begin{tikzcd}
X \ar{r}{\delta} & X^1 \ar{r}{\iota_1} & PX,
\end{tikzcd}\]
which we also denote by $\delta$.
\end{deph}

Proposition~\ref{isodelta} implies that $\delta$ is an isometric embedding. As a composite of natural transformations, we also have naturality $\delta: \1 \Rightarrow P$. Before getting to the multiplication, we need another bit of preparation. As we show in Corollary~\ref{sifted_exist}, $\cat{CMet}$ has sifted colimits. These colimits are preserved by the power functors:

\begin{lemma}
Both the power functors $(-)^S$ and the symmetrized power functors $(-)_n$ preserve sifted colimits in $\cat{CMet}$.
\label{powercocont}
\end{lemma}

\begin{proof}
Let $\cat{D}$ be a sifted category. Since $(-)^S$ is $(-)^{\otimes S}$ composed with a rescaling, it is enough to show that $(-)^{\otimes S}$ preserves $\cat{D}$-colimits. But since the monoidal product preserves colimits in each argument, $(-)^{\otimes S}$ turns a $\cat{D}$-colimit into a $\cat{D}^{\times S}$-colimit. But since the diagonal functor $\cat{D}\to\cat{D}^{\times S}$ is final by the siftedness assumption, the claim for $(-)^S$ follows. The claim for $(-)_n$ follows by commutation of colimits with colimits.
\end{proof}

$\cat{N}^\op$ is trivially sifted thanks to being directed. However, the category $\cat{FinUnif}^\op$ itself is not sifted: for example, the spans in $\cat{FinUnif}$
\[
\begin{tikzcd}[column sep=small,row sep=small]
	& S \idar{dl} \idar{dr} & & & & S \idar{dl} \ar{dr}{\alpha} \\
	S & & S & & S & & S
\end{tikzcd}
\]
are not connected by any zig-zag, for any $S\in\cat{FinUnif}$ with a non-identity automorphism $\alpha:S\to S$.

Similarly to the quotient maps $(X^m)^n\to (X_m)_n$ in~\eqref{doublequotient}, we have a commutative square
\begin{equation}\label{onlyonei}
\begin{tikzcd}
	(X_m)_n \ar{r}{(\iota_m)_n} \ar{d}{\iota_n} & (PX)_n \ar{d}{\iota_n} \\
	P(X_m) \ar{r}{P\iota_m} & PPX
\end{tikzcd}
\end{equation}
where now all maps are isometric embeddings. In the following, we use this composite as the map $(X_m)_n\to PX$.

\begin{prop}
$PPX$ is the colimit of both
\begin{enumerate}
\item the $(X_m)_n$ with the colimiting cocone given by the $\iota_n\circ (\iota_m)_n = P\iota_m \circ \iota_n$ for $m,n\in\cat{N}^\op$;
\item the subdiagram of this formed by the $(X_n)_n$ for $n\in\cat{N}^\op$.
\end{enumerate}
\label{PPuniv}
\end{prop}

While measures on spaces of measures are often quite delicate to handle, this result gives a concrete way to work with them in terms of finite data only. Although we do not currently have any use for even higher powers of $P$, the analogous statement holds for any $P^n X$.

\begin{proof}
The second claim follows from the first since $\cat{N}^\op$ is sifted. For the first, the Lemma~\ref{powercocont} tells us that the $(\iota_m)_n : (X_m)_n\to (PX)_n$ form a colimiting cocone for each $n$; the claim then follows from the construction of a colimit over $\cat{N}^\op\times\cat{N}^\op$ by first taking the colimit over the first factor and then over the second.
\end{proof}

\begin{lemma}\label{Ecommutes}
For $X\in\cat{CMet}$, there is a unique morphism $E : PPX \to PX$ such that
\begin{equation}\label{Ecommuteseq}\begin{tikzcd}
	(X_m)_n \ar{r}{E_{m,n}} \ar{d} & X_{mn} \ar{d} \\
	PPX \ar{r}{E} & PX
\end{tikzcd}\end{equation}
commutes for all $m,n\in\cat{N}$.
\end{lemma}

The map $E: PPX \to PX$ amounts to taking the \emph{expected distribution}. 

\begin{proof}
This amounts to showing that the $\iota_{mn}\circ E_{m,n}$ form a cocone to which the universal property of Proposition~\ref{PPuniv} applies. Since every morphism in $\cat{N}$ is a divisibility relation, this corresponds to commutativity of the two diagrams
\[
 \begin{tikzcd}[row sep=large]
  (X_m)_n \ar["E_{m,n}"']{d} \ar{rr}{(X_{m|\ell m})_n} && (X_{\ell m})_n \ar{d}{E_{\ell m,n}} && (X_m)_n \ar["E_{m,n}"']{d} \ar{rr}{(X_m)_{n|\ell n}} && (X_m)_{\ell n} \ar{d}{E_{m,\ell n}} \\
  X_{mn} \ar["\iota_{mn}"']{dr} \ar{rr}{X_{mn|\ell m n}} && X_{\ell mn} \ar["\iota_{\ell mn}"]{dl} && X_{mn} \ar{rr}{X_{mn|\ell m n}} \ar["\iota_{mn}"']{dr} && X_{\ell mn} \ar["\iota_{\ell mn}"]{dl} \\
  & PX &&&& PX		
 \end{tikzcd}
\]
for every $\ell\in\cat{N}$. The upper squares commute by naturality of $E$ in its two arguments in $\cat{N}$, and the triangles by Lemma~\ref{imn}.
\end{proof}

$E:PPX\to PX$ is natural in $X$ thanks to the uniqueness, i.e.~we have a natural transformation $E: PP \Rightarrow P$.

Let us show why the map $E$ is exactly the integration map taking the expected distribution. Denote for now by $\tilde{E}$ the usual integration map, i.e.~for all $\mu\in PPX$, let $\tilde{E}\mu\in PX$ be the measure mapping every Lipschitz function $f:X\to\R$ to
\begin{equation*}
 \int_X f \, d(\tilde{E}\mu) := \int_{PX} \left( \int_X f \, dp \right) d\mu(p) ,
\end{equation*}
$\tilde{E}$ is short because the map $p \mapsto \int_X f\, dp$ is. It furthermore makes diagram~\eqref{Ecommuteseq} commute, since for all $\{\{x_{11},\dots,x_{m1}\},\dots,\{x_{1n},\dots,x_{mn}\}\}$ in $(X^M)^N$, by linearity of the integral:
\begin{align*}
 &\int f \, d(\tilde{E}\circ \iota_{n}\circ(\iota_{m})_n \{\{x_{11},\dots,x_{m1}\},\dots,\{x_{1n},\dots,x_{mn}\}\}) \\ 
 &= f(x_{11}) + \dots + f(x_{m1}) + \dots + f(x_{1n}) + \dots + f(x_{mn}) \\
 &= \int f \, d(\iota_{mn}\circ E_{m,n}\{\{x_{11},\dots,x_{m1}\},\dots,\{x_{1n},\dots,x_{mn}\}\}) .
\end{align*}
Therefore, again by uniqueness, $\tilde{E}=E$.

\subsection{Monad axioms}
\label{monadax}

$E$ and $\delta$ satisfy the monad axioms. This can be proven using the universal property and the monoidal properties of the power functors described in~\ref{monpf}.

\begin{thm}\label{monadth}
 $(P,\delta,E)$ is a monad on $\cat{CMet}$. In other words, we have commutative diagrams:
 \begin{equation}\begin{tikzcd}
P \nat{r}{P \delta} \idar{dr} & PP \nat{d}{E} & P \ar[swap,Rightarrow]{l}{\delta P} \idar{dl} \\
& P
\end{tikzcd}\end{equation}
and:
\begin{equation}\begin{tikzcd}
PPP \nat{r}{P E} \nat{d}{E P} & PP \nat{d}{E} \\
PP \nat{r}{E} & P
\end{tikzcd}\end{equation}
\end{thm}

When equipped with this additional structure, we call the Kantorovich functor $P$ the \emph{Kantorovich monad}.

\begin{proof} We already know that $\delta$ and $E$ are natural. Hence we only need to check the commutativity at each object $X\in\cat{CMet}$. Because of the universal property of $P$, $E_n$, $E$ and $\iota$, we have the following.
\begin{enumerate}
 \item The left unit triangle at $X$ is the back face of the following prism:
    \begin{equation}\begin{tikzcd}[column sep=huge]
    & PX \ar{r}{\delta} \idar{dr} & PPX \ar{d}{E} \\
    X_m \ar{ur} \ar[pos=0.65]{r}{(X_m)_{1|n}\circ \delta} \ar[swap]{dr}{X_{m|mn}} & (X_m)_n \ar[crossing over]{ur} \ar{d}{E_{m,n}} & PX \\
    & X_{mn} \ar{ur}
    \end{tikzcd}\end{equation}
    Now: 
    \begin{itemize}
     \item The front face can be decomposed into the following diagram:
     \begin{equation}
      \begin{tikzcd}
       X_m \idar{dr} \ar{r}{\delta} & (X_m)_1 \ar{r}{(X_m)_{1|n}} \ar{d}{E_{m,1}} & (X_m)_n \ar{d}{E_{m,n}} \\
       & X_m \ar[swap]{r}{X_{m|mn}} & X_{mn}
      \end{tikzcd}
     \end{equation}
     which commutes by the left unit diagram of Theorem~\ref{monoidalfunctorsym}, together with naturality of $E_{m,-}$;
     \item The top face can be decomposed into the following diagram:
     \begin{equation}
      \begin{tikzcd}
       X_m \ar{d}{\iota_m} \ar{r}{\delta} & (X_m)_1 \ar{d}{(\iota_m)_1} \ar{r}{(X_m)_{1|n}} & (X_m)_n \ar{d}{(\iota_m)_n} \\
       PX \ar[swap]{r}{\delta} & (PX)_1 \ar[swap]{r}{(PX)_{1|n}} & (PX)_n \ar[swap]{r}{\iota_n} & PPX
      \end{tikzcd}
     \end{equation}
     which commutes by naturality of $\delta$ and $(-)_{1|n}$;
     \item The right face commutes by Lemma~\ref{Ecommutes};
     \item The left bottom face commutes by the naturality of the empirical distribution map.
    \end{itemize}
    The empirical distribution maps are not epic, but across \emph{all} $m,n$ they are \emph{jointly} epic, therefore the back face has to commute as well.
 \item The right unit triangle at $X$ is the back face of the following prism:
    \begin{equation}\begin{tikzcd}[column sep=huge]
    & PX \ar{r}{P\delta} \idar{dr} & PPX \ar{d}{E} \\
    X_n \ar{ur} \ar[pos=0.65]{r}{(X_{1|m}\circ\delta)_n} \ar[swap]{dr}{X_{m|mn}} & (X_m)_n \ar[crossing over]{ur} \ar{d}{E_{m,n}} & PX \\
    & X_{mn} \ar{ur}
    \end{tikzcd}\end{equation}
    Now: 
    \begin{itemize}
     \item The front face can be decomposed into the following diagram:
     \begin{equation}
      \begin{tikzcd}
       X_n \idar{dr} \ar{r}{(\delta)_n} & (X_1)_n \ar{r}{(X_{1|m})_n} \ar{d}{E_{1,n}} & (X_m)_n \ar{d}{E_{m,n}} \\
       & X_n \ar[swap]{r}{X_{n|mn}} & X_{mn}
      \end{tikzcd}
     \end{equation}
     which commutes by the right unit diagram of Theorem~\ref{monoidalfunctorsym}, together with naturality of $E_{-,n}$;
     \item The top face can be decomposed into the following diagram:
     \begin{equation}
      \begin{tikzcd}
       X_n \ar{d}{\iota_n} \ar{r}{(\delta)_n} & (X_1)_n \ar{d}{\iota_n} \ar{r}{(X_{1|m})_n} & (X_m)_n \ar{d}{\iota_n} \\
       PX \ar[swap]{r}{P\delta} & P(X_1) \ar[swap]{r}{P(X_{1|m})} & P(X_m) \ar[swap]{r}{P\iota_m} & PPX
      \end{tikzcd}
     \end{equation}
     which commutes by naturality of $\iota_n$;
     \item The right face commutes again by Lemma~\ref{Ecommutes};
     \item The left bottom face commutes again by the naturality of the empirical distribution map.
    \end{itemize}
    Again, the empirical distribution maps across all $m,n$ are jointly epic, therefore the back face has to commute as well.
 \item The associativity square at each $X$ is the back face of the following cube:
    \begin{equation}\begin{tikzcd}
    & PPPX \ar{rr}{PE} \ar[near start]{dd}{E} && PPX \ar{dd}{E} \\
    ((X_\ell)_m)_n \ar{dd}{E_{m,n}} \ar{ur} \ar[crossing over, near end, swap]{rr}{(E_{\ell, m})_n} && (X_{\ell m})_n \ar{ur}  \\
    & PPX \ar[near end]{rr}{E} && PX \\
    (X_\ell)_{mn} \ar{ur} \ar{rr}{E_{\ell, mn}} && X_{\ell m n} \ar{ur} \ar[leftarrow,crossing over, near end]{uu}{E_{\ell m, n}}
    \end{tikzcd}\end{equation}   
    where the map $((X_\ell)_m)_n \to PPPX$ is uniquely obtained in the same way as the map $((X^\ell)^m)^n \to ((X_\ell)_m)_n$ in the proof of Theorem~\ref{monoidalfunctorsym}, using naturality of $\iota$ instead of $q$.
    Now:
    \begin{itemize}
     \item The front face is just the associativity square of Theorem~\ref{monoidalfunctorsym};
     \item The top face can be decomposed into:
     \begin{equation}
      \begin{tikzcd}
       ((X_\ell)_m)_n \ar{d}{(E_{\ell, m})_n} \ar{r} & (PPX)_n \ar{r}{\iota_n} \ar{d}{(E)_n} & PPPX \ar{d}{PE}\\
       (X_{\ell m})_{n} \ar{r} & (PX)_n \ar{r}{\iota_n} & PPX
      \end{tikzcd}
     \end{equation}
     which commutes by Lemma~\ref{Ecommutes}, and by naturality of $\iota_n$;
     \item The left, right, and bottom faces commute by Lemma~\ref{Ecommutes}.
    \end{itemize}
    Once again, the empirical distribution maps across all $\ell,m,n$ are jointly epic, therefore the back face has to commute as well.
\end{enumerate}

It follows that $(P,\delta,E)$ is a monad. 
\end{proof}

In analogy with Theorem~\ref{FinUniftoNmonad}, we can now conclude that $P$ is the monad that one obtains by taking the $\cat{FinUnif}^\op$-graded monad $(=)^{(-)}$ or the $\cat{N}^\op$-graded monad $(=)_{(-)}$ and ``crushing them down'' universally to an ungraded monad:

\begin{thm}
\label{Paskanext}
As a lax monoidal functor, $P:\cat{1}\to[\cat{CMet},\cat{CMet}]$ is the Kan extension in $\cat{MonCat}$
\begin{enumerate}
\item of the strong monoidal functor $(=)^{(-)}:\cat{FinUnif}^\op\to[\cat{CMet},\cat{CMet}]$ along $!:\cat{FinUnif}^\op\to\cat{1}$, and 
\item of the lax monoidal functor $(=)_{(-)}:\cat{N}^\op\to[\cat{CMet},\cat{CMet}]$ along $!:\cat{N}^\op\to\cat{1}$, 
\end{enumerate}
with respect to the empirical distribution maps making up the universal transformation.
\end{thm}

Together with Corollary~\ref{empdcolimit} and Theorem~\ref{FinUniftoNmonad}, this means that we have a diagram
\[
 \begin{tikzcd}
\cat{FinUnif}^\op \ar[swap]{dd}{|-|} \ar{ddrr}{(=)^{(-)}} 
 \ar[bend right=10,""{name=PUP},phantom]{ddrr} \ar[bend right=50,""{name=DPDOWN,pos=0.32},phantom]{ddrr} 
  \ar[Rightarrow,from=PUP,to=DPDOWN,"q",swap] \\
\\
\cat{N} \ar[swap]{dd}{!} \ar{rr}{(=)_{(-)}} 
 \ar[bend right=10,""{name=PDOWN},phantom]{rr} \ar[bend right=90,""{name=P,pos=0.3},phantom]{rr} 
  \ar[Rightarrow,from=PDOWN,to=P,"\iota_{(-)}"]
 && \left[\cat{CMet},\cat{CMet}\right] \\
\, \\
\cat{1} \ar[swap]{uurr}{P}
\end{tikzcd}
\]
in which all 2-cells are Kan extensions, both in $\cat{Cat}$ and in $\cat{MonCat}$.

\begin{proof}
By composition of Kan extensions and Theorem~\ref{FinUniftoNmonad}, it is enough to prove the second item. In order to apply Theorem~\ref{monoidalkanext}, it remains to check two things: first, that the transformation $\iota_{(-)}:(=)_{(-)}\Rightarrow P$ is monoidal, which boils down to the diagram
\[\begin{tikzcd}
	(X_m)_n \ar{r}{E_{m,n}} \ar[swap]{d}{\iota_n\circ (\iota_m)_n} & X_{mn} \ar{d}{\iota_{mn}} \\
	PPX \ar{r}{E} & PX
\end{tikzcd}\]
which equals~\eqref{Ecommuteseq} again. Second, that $i\otimes i$ is an epimorphism in the functor category $[\cat{N}^\op\times\cat{N}^\op,[\cat{CMet},\cat{CMet}]]$, which follows from the fact that for every $X$, the maps $(X_m)_n\to PPX$ are jointly epic.
\end{proof}

The uniqueness statement in Theorem~\ref{monoidalkanext} also shows that the monad structure on $P$ is the only one which makes the empirical distribution maps into a morphism of graded monads.

\section{Algebras of the Kantorovich monad}
\label{Palgssec}

A $P$-algebra for the Kantorovich monad $P$ consists of $A\in \cat{CMet}$ together with a map $e:PA\to A$ such that the following diagrams commute:
    \[\begin{tikzcd}
    A \ar{r}{\delta} \idar{dr} & PA \ar{d}{e} & & & PPA \ar{d}{E} \ar{r}{Pe} & PA \ar{d}{e} \\
    & A & & & PA \ar{r}{e} & A
    \end{tikzcd}\]
The expectation value of a real- or vector-valued random variable is of central importance in probability theory. Indeed, as is generally the case in categorical probability, this is precisely what the map $e$ encodes: $P\R$ is the space of distributions of real-valued random variables with well-defined expectation value, and the operation of taking the expectation is a morphism $P\R\to\R$ which makes $\R$ into a $P$-algebra. Replacing $\R$ by other $P$-algebras results in a definition of integral or expectation for $A$-valued random variables $X\to A$ by composing the induced map $PX \to PA$ with the algebra map $PA \to A$. In this sense, the characterization of $P$-algebras of Theorem~\ref{Palgthm} together with our purely categorical construction of the Kantorovich monad $P$ via Theorem~\ref{Paskanext} below can be regarded as a definition of integral which does not require any measure theory.

A morphism of $P$-algebras $e_A : PA\to A$ and $e_B:PB\to B$ is a short map $f:A\to B$ such that
 \[\begin{tikzcd}
  PA \ar{d}{e_A} \ar{r}{Pf} & PB \ar{d}{e_B} \\
  A \ar{r}{f} & B
 \end{tikzcd}\]
commutes. We also say that $f$ is \emph{$P$-affine}. The Eilenberg-Moore category $\cat{CMet}^P$ is then the category of $P$-algebras and $P$-affine maps. Any Wasserstein space $PX$ is a free $P$-algebra, with structure map $e=E:PPX\to PX$. The Kleisli category $\cat{CMet}_P$ is the full subcategory of $\cat{CMet}^P$ on the free algebras. Its morphisms are the short maps $X\to PY$ for complete metric spaces $X$,$Y$, which correspond bijectively and naturally to $P$-affine maps $PX\to PY$, so that it naturally contains $\cat{CMet}$ as a subcategory.

As usual in categorical probability, the Kleisli morphisms should be thought of as \emph{stochastic maps} or \emph{Markov kernels}~\cite{giry}. An important difference between other approaches to categorical probability theory and the one developed by van Breugel~\cite{breugel} and in this work is that these stochastic maps are also required to be short. This leads to the unpleasant phenomenon that conditional expectations do not always exist: for given $p\in PX$ and $f:X\to Y$, it is generally not possible to write $p$ as the image of the pushforward $(Pf)(p)$ under a Kleisli morphism $PY\to PX$.

In this section, we will give equivalent characterizations of the $P$-algebras and their category. 
In the context of compact and of 1-bounded complete metric spaces, it seems to be known that the Kantorovich monad captures the operations of taking formal binary midpoints \cite[Section~4]{BHMW}. We develop similar ideas for all complete metric spaces.

By evaluating the structure map on a finitely supported measure, one assigns to every formal convex combination of points another point. In this way, a $P$-algebra looks like a convex set in which the convex structure interacts well with the metric. And indeed, we will show that the category of $P$-algebras is equivalent to the category of closed convex subsets of Banach spaces with short affine maps. A similar characterization of the category of algebras of the Radon monad exists, as the category of compact convex sets in locally convex spaces~\cite{swirszcz}; see also~\cite{keimel} for a more recent exposition. A similarly simple characterization of the algebras of the Giry monad is apparently not known; even for the Giry monad on the category of Polish spaces with continuous maps, the existing characterization is significantly more complicated~\cite{doberkat}.

\subsection{Convex spaces}
\label{secconvspaces}

A set together with an abstract notion of convex combination which satisfies the same equations as convex combinations in a vector space is a \emph{convex space}. This is a notion which has been discovered many times over in various forms, e.g.~in~\cites{stone,gudder,swirszcz}. A convex space can be defined as an algebra of the convex combinations monad on $\cat{Set}$. This monad assigns to every set $M$ the set of finitely supported probability measures on $M$, where the unit is again given by the Dirac delta embedding and the multiplication by the formation of the expected measure,
\[
	\sum_i \alpha_i\, \delta\left(\sum_j \beta_{ij} \delta(x_{ij})\right) \longmapsto \sum_{i,j} \alpha_i \beta_{ij} \delta({x_{ij}})
\]
Equivalently, a convex space is a model of the Lawvere theory opposite to the category of stochastic matrices, $\cat{FinStoch}$~\cite{fritz}. An axiomatization in terms of binary operations is as follows:

\begin{deph}
\label{csdef}
A \emph{convex space} is a set $A$ equipped with a family of binary operations $c:[0,1]\times A\times A\to A$, such that the following properties hold for all $x,y,z\in A$ and $\lambda,\mu\in [0,1]$:
 \begin{enumerate}
  \item\label{csunit} Unitality: $c_{0}(x,y) = x$;
  \item\label{csidempo} Idempotency: $c_\lambda(x,x) = x$;
  \item Parametric commutativity: $c_\lambda(x,y) = c_{1-\lambda}(y,x)$;
  \item Parametric associativity: $c_\lambda(c_\mu(x,y),z) = c_{\lambda\mu}(x,c_\nu(y,z))$, where:
  \begin{equation}
  \label{csassoc}
   \nu = \begin{cases}
         \frac{\lambda(1-\mu)}{1-\lambda\mu} & \mbox{ if } \lambda,\mu\ne 1;\\
         \mbox{any number in } [0,1] & \mbox{ if } \lambda=\mu=1.
         \end{cases}
  \end{equation}
 \end{enumerate}
The category of convex spaces has as morphisms the maps $f:A\to B$ for which
 \begin{equation}
  \begin{tikzcd}
   A\times A \ar{d}{c_\lambda} \ar{r}{f\times f} & B \times B \ar{d}{c_\lambda} \\
   A \ar{r}{f} & B
  \end{tikzcd}
 \end{equation}
commutes for every $\lambda\in[0,1]$.
\end{deph}
  
In the following, we will make use of the standard equivalence between this definition of convex spaces and algebras of the convex combinations monad $C:\cat{Set}\to\cat{Set}$.

\subsection{Equivalent characterizations of $P$-algebras}
\label{convalg}

\begin{thm}
\label{Palgthm}
The following structures are equivalent on a complete metric space $A$, in the sense that there is an isomorphism of categories over $\cat{CMet}$:
 \begin{enumerate}
  \item\label{Palg} A $P$-algebra structure;
  \item\label{sympoweralg} A short map $e_n:A_n\to A$ for each $n\in\cat{N}$, such that $e_1 = \delta^{-1}$, and such that the diagrams 
  \begin{equation}
   \begin{tikzcd}
   \label{Nalg}
    A_m \ar[swap]{dr}{e_m} \ar{rr}{A_{m|mn}} && A_{mn} \ar{dl}{e_{mn}} \\
    & A
   \end{tikzcd} \qquad
   \begin{tikzcd}
    (A_m)_n \ar{d}{E_{m,n}} \ar{r}{(e_m)_n} & A_n \ar{d}{e_n} \\
    A_{mn} \ar{r}{e_{mn}} & A
   \end{tikzcd}
  \end{equation}
  commute. Structure-preserving maps are those $f:A\to B$ for which the diagrams
  \begin{equation}
   \label{Nalgmorph}
   \begin{tikzcd}
    A_n \ar{d}{e_n} \ar{r}{f_n} & B_n \ar{d}{e_n} \\
    A \ar{r}{f} & B
   \end{tikzcd}
  \end{equation}
  commute for all $n\in\cat{N}$.
  \item\label{poweralg} A short map $e^S:A^S\to A$ for each $S\in\cat{FinUnif}$, such that $e^1 = \delta^{-1}$, and such that the diagrams 
  \begin{equation}
   \begin{tikzcd}
   \label{FinUnifalg}
    A^T \ar[swap]{dr}{e^T} \ar{rr}{A^\phi} && A^S \ar{dl}{e^S} \\
    & A
   \end{tikzcd} \qquad
   \begin{tikzcd}
    (A^S)^T \ar{d}{E^{S,T}} \ar{r}{(e^S)^T} & A^T \ar{d}{e^T} \\
    A^{S\times T} \ar{r}{e^{S\times T}} & A
   \end{tikzcd}
  \end{equation}
  commute for every $S,T\in\cat{FinUnif}$ and $\phi\in\cat{FinUnif}(S,T)$. Structure-preserving maps are those $f:A\to B$ for which the diagrams
  \begin{equation}
   \label{FinUnifalgmorph}
   \begin{tikzcd}
    A^S \ar{d}{e^S} \ar{r}{f^S} & B^S \ar{d}{e^S} \\
    A \ar{r}{f} & B
   \end{tikzcd}
  \end{equation}
  commute for all $S\in\cat{FinUnif}$.
  \item\label{csmet} A structure of convex space satisfying a compatibility inequality with the metric,
\begin{equation}
\label{metcompat}
   d\big( c_\lambda(x, z), c_\lambda(y, z) \big) \le \lambda \, d(x,y),
\end{equation}
where the morphisms are the short maps that are also morphisms of convex spaces.
 \end{enumerate}
\end{thm}

We make two remarks on related literature. First, in the special case of complete separable metric spaces,~\cite[Theorem~10.9]{mpp}\footnote{Modulo the earlier footnote \ref{mpp_error}.} can also be interpreted as establishing the equivalence between~\ref{Palg} and~\ref{csmet}. Second, the structures in~\ref{poweralg} and~\ref{sympoweralg} are different from the \emph{graded algebras} in the sense of~\cite[Definition~1]{ssm}: for a graded algebra, the algebra morphisms would have to be of type $(A_m)_n\to A_{mn}$ and $(A^S)^T\to A^{S\times T}$, respectively.

It follows from Theorem~\ref{algban} that in structures of type~\ref{csmet}, the inequality~\eqref{metcompat} necessarily holds with equality.

\begin{proof}
We first show that the structures of type~\ref{Palg}, \ref{poweralg} and~\ref{sympoweralg} are equivalent, using the universal property proven in~\ref{uniprop}.

\begin{itemize}
\item \ref{sympoweralg}$\Leftrightarrow$\ref{poweralg}: By composing with the quotient maps $A^S\to A_{|S|}$, the $(e_n)_{n\in\cat{N}}$ determine morphisms $e^S : A^S\to A$, and conversely by the universal property. The equivalence between the triangles in~\eqref{Nalg} and~\eqref{FinUnifalg} follows from $e_{|S|} = e^S \circ q$ and the diagram~\eqref{imnquot}. For the same reason, \eqref{Nalgmorph} and~\eqref{FinUnifalgmorph} are equivalent.

It remains to verify the equivalence of the squares in~\eqref{Nalg} and~\eqref{FinUnifalg}. This follows by a cube similar to~\eqref{symcube},
\[
\begin{tikzcd}[row sep=large,column sep=large]
        & (A_m)_n \ar{rr}{(e_m)_n} \ar["E_{m,n}" near start]{dd} & & A_n \ar{dd}{e_n} \\
	(A^m)^n \ar[crossing over,"(e^m)^n" near start]{rr} \ar{dd}{E^{m,n}} \ar{ur} & & A^n \ar{ur} \\
        & A_{mn} \ar["e_{mn}" near end]{rr}  & & A \\
	A^{mn} \ar{rr}{e^{mn}} \ar{ur} & & A \idar{ur} \ar[leftarrow, swap, crossing over,"e^n"' near end]{uu}
\end{tikzcd}
\]
where the front face commutes if and only if the back face commutes, since all other faces commute, the quotient map on the upper left is epic, and the identity on the lower right is monic.
\item \ref{Palg}$\Leftrightarrow$\ref{sympoweralg}: This works similarly. By the universal property of $PA$, the cocone defined by the first diagram in~\eqref{Nalg} is equivalent to a short map $e:PA\to A$. The equivalence between the square in~\eqref{Nalg} and the multiplication square of a $P$-algebra follow by considering the cube
    \begin{equation}\begin{tikzcd}
    & PPA \ar{rr}{Pe} \ar[near start]{dd}{E} && PA \ar{dd}{e} \\
    (A_m)_n \ar{dd}{E_{m,n}} \ar{ur} \ar[crossing over, near end, swap]{rr}{(e_{m})_n} && A_n \ar{ur}  \\
    & PA \ar[near end]{rr}{e} && A \\
    A_{mn} \ar{ur} \ar{rr}{e_{mn}} && A \idar{ur} \ar[leftarrow,crossing over, near end]{uu}{e_n}
    \end{tikzcd}\end{equation} 
and using that the upper left diagonals are jointly epic as $m$ and $n$ vary.
 \item \ref{Palg}$\Rightarrow$\ref{csmet}: Finite convex combinations with real coefficients are a special case of Radon measures, and therefore every $P$-algebra $e:PA\to A$ also is a convex space in a natural way. Technically, this is based on the morphism of monads
\[\begin{tikzcd}
	\cat{CMet} \ar{d}{U} \ar{r}{P} & \cat{CMet} \ar{d}{U} \\
	\cat{Set} \nat{ur}{\eta} \ar{r}{C} & \cat{Set}
\end{tikzcd}\]
where $U$ is the forgetful functor, and $\eta$ is the natural transformation with $\eta:CUX\to UPX$ given by the map which reinterprets a finitely supported measure on $UX$ as a finitely supported measure on $X$, considered as an element of the underlying set of $PX$. It is straightforward to check that this is a morphism of monads. Thus we have a functor from $P$-algebras in $\cat{CMet}$ to $C$-algebras in $\cat{Set}$. In other words, every $P$-algebra is a convex space in a canonical way.
 
 Let us now check compatibility with the metric. Since $e$ is short, we get
 \begin{align*}
  d\big( c_\lambda(x& ,z), c_\lambda(y,z) \big) \\
  &= d\left(e(\lambda\delta(x) + (1-\lambda)\delta(z)), e(\lambda\delta(y) + (1-\lambda)\delta(z)) \right) \\
  &\le d\left(\lambda\delta(x) + (1-\lambda)\delta(z), \lambda\delta(y) + (1-\lambda)\delta(z) \right) \\
  &\stackrel{\textrm{Lemma~\ref{dcontract}}}{=} \lambda\, d(\delta(x),\delta(y)) = \lambda\, d(x,y).
 \end{align*}
 \item \ref{csmet}$\Rightarrow$\ref{poweralg}: Intuitively, the $e^S$ correspond to taking convex combinations with equal weights, and commutativity of~\eqref{FinUnifalg} follow from the equations satisfied by taking convex combinations in any convex space. To make this precise, it is most convenient to consider a convex space as a model of the Lawvere theory $\cat{FinStoch}^\op$. Considering $\cat{FinUnif}$ as a subcategory $\cat{FinUnif}\subseteq\cat{FinSet}\subseteq\cat{FinStoch}$, defining maps $u_S:1\to S$ in $\cat{FinStoch}$ which select the uniform distribution on each finite set $S$ results in commutativity of the two diagrams
\[\begin{tikzcd}
	& 1 \ar["u_S"']{ld} \ar{rd}{u_T} \\
	S \ar{rr}{\phi} & & T
\end{tikzcd}\qquad
\begin{tikzcd}
	1 \ar{r}{u_{S\times T}} \ar{d}{u_T} & S\times T \idar{d} \\
	T \ar{r}{u_S \times T} & S \times T
\end{tikzcd}\]
for every $S,T\in\cat{FinUnif}$ and $\phi\in\cat{FinUnif}(S,T)$. Thus, given a convex space $A$ as a model of $\cat{FinStoch}^\op$, the $u_S$ become maps $e^S:A^S\to A$ satisfying the required equations, and every affine map between convex spaces will make~\eqref{FinUnifalgmorph} commute. To see that the $e^S$ are short apply~\eqref{metcompat} twice, to get
\[
   d\big( c_\lambda(x, z), c_\lambda(y, w) \big) \le \lambda \, d(x,y) + (1-\lambda)\, d(z,w),
\]
which generalizes to
\begin{equation}
\label{metcompatgeneral}
	d\left(e\left(\sum_i \lambda_i \delta(x_i)\right), e\left(\sum_i \lambda_i \delta(y_i) \right) \right) \le \sum_i \lambda_i d(x_i,y_i)
\end{equation}
by decomposing a general convex combination into a sequence of binary ones and by induction. Shortness of $e^S$ is the special case where the $\lambda_i$'s are uniform and equal to $1 / |S|$.
\end{itemize}
It is clear that, starting with a $P$-algebra $A$ and applying the constructions \ref{Palg}$\Rightarrow$\ref{csmet}$\Rightarrow$\ref{poweralg}, one recovers the underlying \ref{poweralg}-structure of $A$. To see that the composite functor \ref{csmet}$\Rightarrow$\ref{sympoweralg}$\Leftrightarrow$\ref{Palg}$\Rightarrow$\ref{csmet} is the identity as well, we claim that two convex space structures $c$ and $c'$ which satisfy the metric compatibility inequality and coincide for convex combinations with rational weights must be equal. Indeed, we prove $d(c_\lambda(x,y),c'_\lambda(x,y)) = 0$ for all $\lambda\in(0,1)$: first, as $\lambda$ varies, this distance is bounded because $d(c_\lambda(x,y),y) = d(c_\lambda(x,y),c_\lambda(y,y)) \leq \lambda d(x,y) \le d(x,y)$, and similarly for $c'$, so that we get an upper bound of $2 d(x,y)$,
\[
	d(c_\mu(x,y),c'_\mu(x,y)) \leq 2 d(x,y) \quad \forall\mu\in[0,1].
\]
We choose a sufficiently small rational $\varepsilon > 0$, as well as a rational $\nu\in\left(\tfrac{\lambda - \varepsilon}{1 - \varepsilon},\tfrac{\lambda}{1 - \varepsilon}\right)$, and put $z:=c_\nu(x,y) = c'_\nu(x,y)$. Then
\[
	c_\lambda(x,y) = c_\varepsilon(c_\mu(x,y),z), \qquad c'_\lambda(x,y) = c'_\varepsilon(c'_\mu(x,y),z),
\]
where $\mu:=\tfrac{\lambda - (1-\varepsilon)\nu}{\varepsilon}$ is in $[0,1]$ due to the assumed bounds on $\nu$. Now since $\varepsilon$ is rational, we can bound the distance between these two points by
\begin{align*}
	d(c_\lambda(x,y),c'_\lambda(x,y)) & = d(c_\varepsilon(c_\mu(x,y),z),c_\varepsilon(c'_\mu(x,y),z)) \\
	 & \leq \varepsilon\, d(c_\mu(x,y),c'_\mu(x,y)) \leq 2\varepsilon\, d(x,y),
\end{align*}
from which the claim follows as $\varepsilon \to 0$.
\end{proof}

We suspect that $P$-algebras also coincide with the \emph{metric mean-value algebras} of~\cite[Definition~6]{BHMW}, if the requirement of $1$-boundedness is dropped.

\subsection{Algebras as closed convex subsets of Banach spaces}
\label{banach}

If $(E,\|\cdot\|)$ is a Banach space and $A\subseteq E$ is a closed convex subset, then $A$ is a convex space which carries the metric
\[
	d(x,y) := \|x-y\|
\]
with respect to which it is complete. These two structures fulfill the metric compatibility inequality~\eqref{metcompat},
\[
	\|(\lambda x + (1-\lambda)z) - (\lambda y + (1-\lambda)z)\| = \|\lambda x - \lambda y\| = \lambda \|x - y\|,
\]
which even holds with equality. Therefore, by Theorem~\ref{convalg}\ref{csmet}, $A$ is a $P$-algebra $e:PA\to A$ in a canonical way. In particular, we can \emph{define} the expectation value $\int_A x \, dp(x)$ of any $p\in PA$ (which has finite first moment) to be $e(p)$. By functoriality of $P$, this also defines the expectation value of any Banach-space valued random variable with finite first moment on any (other) complete metric space.

Let $\cat{ConvBan}$ be the category whose objects are closed convex subsets of Banach spaces $A\subseteq E$, and whose morphisms $f:(A\subseteq E)\to (B\subseteq F)$ are the short affine maps $f:A\to B$.\footnote{One might be tempted to define morphisms to be equivalence classes of short affine maps $f:E\to F$ which satisfy $f(A)\subseteq B$, where two such maps are identified whenever they are equal on $A$. This is not equivalent, since a short affine map $A\to F$ can in general not be extended to a short (or even merely continuous) affine map $E\to F$.} We have a canonical functor $\cat{ConvBan}\to\cat{CMet}^P$ which is fully faithful.

Moreover, it was shown in~\cite{capraro} that this functor is essentially surjective, meaning that every $P$-algebra in the form~\ref{csmet} is isomorphic both as a convex space and as a metric space (or equivalently as a $P$-algebra) to a closed convex subset of a Banach space. We therefore conclude that $P$-algebras and closed convex subsets of Banach spaces are the same concept:

\begin{thm}\label{algban}
 The functor $\cat{CBan}\to\cat{CMet}^P$ is an equivalence of categories.
\end{thm}

\appendix

\section{Colimits of (complete) metric spaces}
\label{colims_met_app}

The main goal of this section is to prove that the monoidal structures of $\cat{Met}$ and $\cat{CMet}$ preserve colimits (Proposition~\ref{tensor_cocont}).

To this end, it is useful to consider $\cat{Met}_\infty$, by which we mean the category defined just as $\cat{Met}$ above, with the only difference that we now allow distances in $[0,\infty]$. Given $(X,d)\in\cat{Met}_\infty$, the ``having finite distance'' relation partitions the set $X$ into equivalence classes, each of which is a metric space with finite distances and therefore an object of $\cat{Met}$. The morphisms of $\cat{Met}_\infty$ must respect this equivalence relation. By this line of reasoning, it is straightforward to verify that $\cat{Met}_\infty$ is a free coproduct completion of $\cat{Met}$, with the inclusion functor $\cat{Met} \to \cat{Met}_\infty$ as the universal arrow. The advantage of $\cat{Met}_\infty$ over $\cat{Met}$ is that $\cat{Met}_\infty$ is cocomplete, with colimits of the following form. Let $U : \cat{Met}_\infty \to \cat{Set}$ be the underlying set functor and $F : \cat{C} \to \cat{Met}_\infty$ be a diagram indexed by a small category $\cat{C}$. Equip the set $\colim\: (U\circ F)$ with the largest pseudometric (with possibly infinite values) which makes the colimiting cocone components $\ell_c : U(F(c)) \to \colim\: (U\circ F)$ short; this pseudometric can be computed by starting with the distance function
\begin{equation}
\label{precolim}
	\hat{d}(x,y) := \inf \: \{\: d_{F(c)}(x',y') \: :\: c\in\cat{C}, \: \ell_c(x') = x, \: \ell_c(y') = y \:\},
\end{equation}
and then by shortest path optimization on $\hat{d}$ in order to enforce the triangle inequality, resulting in a pseudometric on $\colim\: (U\circ F)$ (with possibly infinite values). Identifying all pairs of points with distance zero results in an object of $\cat{Met}_\infty$, and it is not hard to see that this is the colimit of $F$ in $\cat{Met}_\infty$.

The very same statements apply to the category $\cat{CMet}_\infty$, defined as the full subcategory of $\cat{CMet}$ on the complete spaces, and its full subcategory $\cat{CMet}$. Since $\cat{CMet}_\infty \subseteq \cat{Met}_\infty$ is also a reflective subcategory of $\cat{Met}_\infty$ with the completion as the reflector, colimits in $\cat{CMet}_\infty$ can be computed as colimits in $\cat{Met}_\infty$ using the procedure above and then taking the completion.

\begin{lemma}
\label{colims_met_lem}
A functor $F : \cat{C} \to \cat{Met}$ (resp.~$F : \cat{C} \to \cat{CMet}$) for a small category $\cat{C}$ has a colimit if and only if the full subcategory on those objects $c\in\cat{C}$ with $F(c) \not\cong \emptyset$ is connected.
\end{lemma}

Intuitively, the only obstruction to the existence of colimits is the lack of coproducts.

\begin{proof}
We only treat the case of $\cat{Met}$, since the case of $\cat{CMet}$ is completely analogous.

If $F : \cat{C} \to \cat{Met}$ satisfies the connectedness assumption, then the shortest path optimization for~\eqref{precolim} will clearly have a finite result, so that the resulting colimit in $\cat{Met}_\infty$ is actually an object in $\cat{Met}$, and is a colimit in there as well since $\cat{Met}\subseteq\cat{Met}_\infty$ is full.

Conversely, suppose that $F : \cat{C} \to \cat{Met}$ does not satisfy the connectedness assumption. By restricting $F$ to the full subcategory $\cat{C}'$ on those $c\in\cat{C}$ with $F(c)\neq\emptyset$, we obtain $F' : \cat{C}' \to \cat{Met}$ such that $F$-cocones and $F'$-cocones are in canonical bijection, thanks to initiality of $\emptyset$. We can therefore assume $F = F'$ without loss of generality, or equivalently $F(c) \neq \emptyset$ for all $c$. Then the non-connectedness assumption also means that the category $\cat{C}$ is not connected, and therefore can be partitioned into disjoint subcategories $\cat{C}_1$ and $\cat{C}_2$ without any morphisms between them, and inhabited by objects $c_1\in\cat{C}_1$ and $c_2\in\cat{C}_2$ with points $x_1\in F(c_1)$ and $x_2\in F(c_2)$. Considering the space $Y := \{1,2\}$ with $d(1,2) = r \in \R_+$ a parameter, there is a cocone from $F$ to $Y$ which takes all points in the spaces assigned to $\cat{C}_1$ to $1\in Y$, and all points in the spaces assigned to $\cat{C}_2$ to $2\in Y$. Hence if the colimit existed in $\cat{Met}$, then the distance of the images of $x_1$ and $x_2$ in the colimit space would have to be lower bounded by $d(1,2) = r$. This is absurd since $r$ was arbitrary.
\end{proof}

The characterization of sifted categories as inhabited categories with connected categories of cospans~\cite[Theorem~1.6]{sifted} now implies:

\begin{cor}
\label{sifted_exist}
Sifted colimits exist both in $\cat{Met}$ and in $\cat{CMet}$.
\end{cor}

In addition, we obtain:

\begin{prop}
The inclusion functor $\cat{Met} \subseteq \cat{Met}_\infty$ (resp.~$\cat{CMet} \subseteq \cat{CMet}_\infty$) preserves all colimits which exist in $\cat{Met}$.
\end{prop}

\begin{proof}
As we saw in the proof of Lemma~\ref{colims_met_lem}, the colimits which exist in $\cat{Met}$ are all also colimits in $\cat{Met}_\infty$, and likewise in the complete case.
\end{proof}

\begin{prop}
For every $X\in\cat{Met}$ (resp.~$X\in\cat{CMet}$), the functor $X\otimes - : \cat{Met} \to \cat{Met}$ (resp.~$X\otimes - : \cat{CMet} \to \cat{CMet}$) preserves colimits.	
\label{tensor_cocont}
\end{prop}

\begin{proof}
We first treat the case of $\cat{Met}$. We note that $\cat{Met}_\infty$ is monoidal closed under the $\ell^1$-product, using the supremum distance for turning the hom-sets $\cat{Met}(X,Y)$ into metric spaces; this is essentially~\cite[Section~2]{lawvere}. Hence $X\otimes - : \cat{Met}_\infty \to \cat{Met}_\infty$ is a left adjoint for every $X\in\cat{Met}$, and as such it preserves colimits, including those of $\cat{Met}$.

The complete case is again analogous, noting that the hom-space $\cat{CMet}_\infty(X,Y) := \cat{Met}_\infty(X,Y)$ with the sup-distance is already complete (by the existence of pointwise limits). This makes $\cat{CMet}_\infty$ also closed, and the inclusion $\cat{CMet}_\infty \subseteq \cat{Met}_\infty$ preserves exponential objects.
\end{proof}

\section{Kan extensions of lax monoidal functors}
\label{kan}

In this section we give a criterion for when a Kan extension of lax monoidal functors is a Kan extension in the bicategory $\cat{MonCat}$ of monoidal categories, lax monoidal functors and monoidal natural transformations. This in particular comprises the Kan extensions used to define the monad structure of $P$ in Section~\ref{secmonad}.

Kan extensions in $\cat{MonCat}$, or more generally, Kan extensions preserving some given algebraic structure, are known in the literature as \emph{algebraic Kan extensions}~\cites{algkanweber,koudenburg}. 
There are some results in the literature on when a left Kan extension in $\cat{Cat}$ of lax or strong monoidal functors is again monoidal (see~\cite[Theorem~1]{mt} and~\cite[Proposition~4]{paterson}) in such a way that the Kan extension also holds in $\cat{MonCat}$. There are also general results on when a Kan extension on a 2-category or double category can be lifted to a Kan extension in the 2-category of pseudoalgebras of a 2-monad (see~\cite[Theorem~1.1b]{koudenburg}, \cite[Theorem~7.7]{hypervirtual} and~\cite[Theorem~2.4.4]{algkanweber}), which can be applied to the monoidal category 2-monad.
In~\cite{ours_kan}, we also derive a statement of this type, which is very close to an instance of Koudenburg's result~\cite[Theorem~7.7]{hypervirtual}. 
For a more detailed discussion of the literature, and for additional context on algebraic Kan extensions, we refer to the discussion in~\cite[Section~1.2]{ours_kan}. 

None of the results cited above applies verbatim to the situation of Section~\ref{secmonad}. Therefore, here we derive a result of this type, tailored to our needs.

For a monoidal category $\cat{C}$, we denote its unit $e:\cat{1}\to\cat{C}$ and multiplication $\otimes:\cat{C}\times\cat{C}\to\cat{C}$ without explicit reference to the category. For a lax monoidal functor $F$, we denote its unit by $\eta_F$ and its multiplication by $\mu_F$.

\begin{thm}\label{month}
\label{monoidalkanext}
Let the following hypotheses be satisfied:

\begin{itemize}
 \item In $\cat{MonCat}$, we have a diagram 
 \begin{equation}
  \begin{tikzcd}
   \cat{C} \ar[swap]{dr}{G} \ar{rr}{F} & \nat{d}{\lambda} & \cat{D} \\
   & \cat{C'} \ar[swap]{ur}{L}
  \end{tikzcd}
 \end{equation}
 \item $\lambda$ makes $L$ into the left Kan extension of $F$ along $G$ in $\cat{Cat}$. 
 \item $G:\cat{C}\to\cat{C'}$ is \emph{strong} monoidal and \emph{essentially surjective}.
 \item The natural transformation $\lambda(-)\otimes\lambda(-)$, by which we mean
 \begin{equation}\label{lambdaepi}
  \begin{tikzcd}
   \cat{C}\times\cat{C} \ar[""{name=FF,below}]{rr}{F\times F} \ar[swap]{dr}{G\times G} && \cat{D}\times\cat{D} \ar{r}{\otimes} & \cat{D} \\
   & \cat{C'}\times\cat{C'} \ar[swap]{ur}{L\times L} \ar[Rightarrow,from=FF,"\lambda\times\lambda"]
  \end{tikzcd}
 \end{equation}
 is an epimorphism in the functor category $\cat{Cat}(\cat{C}\times\cat{C},\cat{D})$.
\end{itemize}
Then $\lambda$ makes $L$ into the left Kan extension of $F$ along $G$ also in $\cat{MonCat}$. Moreover, the monoidal structure of $L$ is the only monoidal structure that can be put on $L$ such that $\lambda$ is monoidal.
\end{thm}

In comparison to previous results, this is closest to~\cite[Theorem~1.1b]{koudenburg}. In fact, Koudenburg's theorem could alternatively be used for the proof of Theorem~\ref{Paskanext}, but not for the proof of Theorem~\ref{FinUniftoNmonad}, for which we really need Theorem~\ref{monoidalkanext}.

\begin{proof}
Given a lax monoidal functor $X:\cat{C'}\to\cat{D}$ and a monoidal transformation $\chi:F\Rightarrow X\circ G$, we can apply the Kan extension property in $\cat{Cat}$, so that there exists a unique $u:L\Rightarrow X$ such that
\begin{equation}\label{uniu}
  \begin{tikzcd}
   \cat{C} \ar[swap]{ddr}{G} \ar{rr}{F} 
    \ar[bend right=10,""{name=F},phantom]{rr} \ar[bend right=90,""{name=GL,pos=0.48},phantom]{rr} 
     \ar[Rightarrow,bend right=20,from=F,to=GL,"\lambda",swap]
    && \cat{D} \\
   \\
   & \cat{C'} \ar[bend left,""{name=L,below}]{uur}{L} \ar[swap,bend right,""{name=X,above}]{uur}{X}
    \ar[Rightarrow,from=L,to=X,"u"]
  \end{tikzcd}
  \quad\equiv\quad
  \begin{tikzcd}
   \cat{C} \ar[swap]{ddr}{G} \ar{rr}{F} 
    \ar[bend right=20,""{name=F},phantom]{rr} \ar[bend right=90,""{name=GX,pos=0.5},phantom]{rr} 
     \ar[Rightarrow,bend left=10,from=F,to=GX,"\chi"]
    && \cat{D} \\
   \\
   & \cat{C'}  \ar[swap,bend right]{uur}{X}
  \end{tikzcd}
 \end{equation}
 What we need to show is that this $u$ is automatically monoidal. We first prove that it respects the units,
\begin{equation}\label{unitu}
 \begin{tikzcd}
  & \cat{C'} \ar[bend left=50,""{name=X,left}]{dd}{X} \ar[bend right=20,""{name=L,pos=0.51,right}]{dd}[pos=0.8]{L}
   \ar[Rightarrow,from=L,to=X,"u"] \\
  \cat{1} \ar{ur}{e} \ar[swap]{dr}{e}
   \ar[bend left=15,""{name=ED},phantom]{dr} \ar[bend left=90,""{name=ECPL,pos=0.38},phantom]{dr} 
    \ar[Rightarrow,bend left=10,from=ED,to=ECPL,"\eta_L"] \\
  & \cat{D} 
 \end{tikzcd}
 \quad \equiv \quad
 \begin{tikzcd}
  & \cat{C'} \ar[bend left=50,""{name=ECPX,below,pos=0.}]{dd}{X} \\
  \cat{1} \ar{ur}{e} \ar[swap]{dr}{e}
   \ar[bend left=15,""{name=ED},phantom]{dr} 
    \ar[Rightarrow,bend right=20,from=ED,to=ECPX,"\eta_X",swap] \\
  & \cat{D} 
 \end{tikzcd}
\end{equation}
To obtain this, we use that $\lambda$ respects units, which means
\begin{equation}\label{unitlambda}
  \begin{tikzcd}
   & \cat{C} \ar[bend left]{dr}{G} \\
   \cat{1} \ar{ur}{e} \ar[near start]{rr}{e} \ar[swap]{dr}{e} 
    \ar[bend left=10,""{name=ECP},phantom]{rr} \ar[bend left=60,""{name=ECG},phantom]{rr}
     \ar[Rightarrow,bend right=20,from=ECP,to=ECG,"\eta_G",swap]
    \ar[bend left=10,""{name=ED},phantom]{dr} \ar[bend left=90,""{name=ECPL,pos=0.6},phantom]{dr}
     \ar[Rightarrow,bend right=20,from=ED,to=ECPL,"\eta_L",swap]
    && \cat{C'} \ar[bend left]{dl}{L} \\
   & \cat{D} 
  \end{tikzcd}
  \quad \equiv \quad
  \begin{tikzcd}
   & \cat{C} \ar[bend left]{dr}{G} \ar[bend right=20]{dd}{F}
    \ar[bend left=10,""{name=F},phantom]{dd} \ar[bend left=90,""{name=GL},phantom]{dd}
     \ar[Rightarrow,from=F,to=GL,"\lambda"] \\
   \cat{1} \ar{ur}{e} \ar[""{name=un, above},swap]{dr}{e} 
    \ar[bend left=90,""{name=ECF,pos=0.4},phantom]{dr} \ar[bend left=10,""{name=ED},phantom]{dr}
     \ar[Rightarrow,bend left=20,from=ED,to=ECF,"\eta_F"]
    && \cat{C'} \ar[bend left]{dl}{L} \\
   &\cat{D} 
  \end{tikzcd}
 \end{equation}
and similarly for $\chi$. Since $\eta_G$ is an isomorphism,~\eqref{unitu} follows if we can prove it after postcomposing with $\eta_G$,
 \begin{align*}
  &\begin{tikzcd}[ampersand replacement=\&, column sep=large, row sep=large]
   \& \cat{C} \ar{dr}{G} \\
   \cat{1} \ar{rr}{e} \ar{ur}{e} \ar[swap,""{name=un,above}]{dr}{e}
    \ar[bend left=60,""{name=ECG},phantom]{rr} \ar[bend left=18,""{name=ECP,pos=0.51},phantom]{rr} 
     \ar[Rightarrow,from=ECP,to=ECG,"\eta_G"{left},"\cong"{right}]
    \ar[bend left=15,""{name=ED},phantom]{dr} \ar[bend left=90,""{name=ECPL,pos=0.68},phantom]{dr} 
     \ar[Rightarrow,bend left=10,from=ED,to=ECPL,"\eta_L",swap]
    \& \& \cat{C'} \ar[bend left,""{name=X,above}]{dl}{X} \ar[bend right,swap,""{name=L,below}]{dl}{L}
     \ar[Rightarrow,from=L,to=X,"u"] \\
   \& \cat{D}
  \end{tikzcd} 
  \quad \equiv \quad
  \begin{tikzcd}[ampersand replacement=\&, column sep=large, row sep=large]
   \& \cat{C} \ar{dr}{G} \ar[swap]{dd}{F} 
    \ar[bend left=10,""{name=F},phantom]{dd} \ar[bend left=90,""{name=GL,pos=0.48},phantom]{dd} 
     \ar[Rightarrow,bend left=20,from=F,to=GL,"\lambda"] \\
   \cat{1} \ar{ur}{e} \ar[swap,""{name=un,above}]{dr}{e} 
    \ar[bend left=15,""{name=ED},phantom]{dr} \ar[bend left=90,""{name=ECF,pos=0.42},phantom]{dr} 
     \ar[Rightarrow,from=ED,to=ECF,"\eta_F"]
    \& \& \cat{C'} \ar[bend left,""{name=X,above}]{dl}{X} \ar[bend right,swap,""{name=L,below}]{dl}[pos=0.6]{L}
     \ar[Rightarrow,from=L,to=X,"u"] \\
   \& \cat{D}
  \end{tikzcd} \\
  \equiv\quad &\begin{tikzcd}[ampersand replacement=\&, column sep=large, row sep=large]
   \& \cat{C} \ar{dr}{G} \ar[swap]{dd}{F} 
    \ar[bend left=20,""{name=F},phantom]{dd} \ar[bend left=90,""{name=GX,pos=0.5},phantom]{dd} 
     \ar[Rightarrow,bend right=10,from=F,to=GX,"\chi"] \\
   \cat{1} \ar{ur}{e} \ar[swap,""{name=un,above}]{dr}{e} 
    \ar[bend left=15,""{name=ED},phantom]{dr} \ar[bend left=90,""{name=ECF,pos=0.42},phantom]{dr} 
     \ar[Rightarrow,from=ED,to=ECF,"\eta_F"]
    \& \& \cat{C'} \ar[bend left,""{name=X,below}]{dl}{X} \\
   \& \cat{D}
  \end{tikzcd}
  \quad \equiv \quad
  \begin{tikzcd}[ampersand replacement=\&, column sep=large, row sep=large]
   \& \cat{C} \ar{dr}{G} \\
   \cat{1} \ar[""{name=ECPX,pos=0.9,below}]{rr}{e} \ar{ur}{e} \ar[swap,""{name=un,above}]{dr}{e} 
    \ar[bend left=60,""{name=ECG},phantom]{rr} \ar[bend left=18,""{name=ECP,pos=0.51},phantom]{rr} 
     \ar[Rightarrow,from=ECP,to=ECG,"\eta_G"{left},"\cong"{right}]
    \ar[bend left=15,""{name=ED},phantom]{dr} 
     \ar[Rightarrow,bend right=15,from=ED,to=ECPX,"\eta_X",swap]
    \&  \& \cat{C'} \ar[bend left]{dl}{X} \\
   \& \cat{D}
  \end{tikzcd}
 \end{align*}
which proves the claim.
 
Proving compatibility with the multiplication
\begin{equation}\label{multu}
 \begin{tikzcd}
  \cat{C'}\times\cat{C'} \ar[bend right=50,swap,""{name=LL,right}]{dd}{L\times L} \ar[bend left=50,""{name=XX,left}]{dd}{X\times X} \ar{rr}{\otimes}
   \ar[bend right=30,""{name=XXT,pos=0.55},phantom]{ddrr} \ar[bend left=70,""{name=TX,pos=0.65},phantom]{ddrr} 
    \ar[Rightarrow, bend right=20, from=XXT,to=TX, "\mu_X",swap]
   && \cat{C'} \ar[bend left=50]{dd}{X} \\
  \\
  \cat{D}\times\cat{D} \ar[swap]{rr}{\otimes} && \cat{D}
  \ar[Rightarrow,from=LL,to=XX,"u\times u"]
 \end{tikzcd}
 \quad \equiv \quad
 \begin{tikzcd}
  \cat{C'}\times\cat{C'} \ar[bend right=50,swap]{dd}{L\times L} \ar{rr}{\otimes}
   \ar[bend right=60,""{name=LLT,pos=0.35},phantom]{ddrr} \ar[bend left=30,""{name=TL,pos=0.4},phantom]{ddrr} 
    \ar[Rightarrow, bend left=20, from=LLT,to=TL, "\mu_L"]
   && \cat{C'} \ar[bend right=50,swap,""{name=L,right}]{dd}{L} \ar[bend left=50,""{name=X,left}]{dd}{X} 
   \ar[Rightarrow,from=L,to=X,"u"] \\
  \\
  \cat{D}\times\cat{D} \ar[swap]{rr}{\otimes} && \cat{D}
 \end{tikzcd}
\end{equation}
 works similarly, but is a bit trickier. We use compatibility of $\lambda$ with the multiplication \begin{equation}\label{multlambda}
  \begin{tikzcd}[column sep=small]
    \cat{C}\times\cat{C} \ar[swap, bend right=50]{dd}{F\times F} 
    \ar[swap, bend right=30,""{name=TF},phantom]{dd} \ar[swap, bend left=55,""{name=TGTL},phantom]{dd}
     \ar[Rightarrow,from=TF,to=TGTL,"\lambda\times\lambda"]
    \ar{rr}{\otimes} \ar[bend left,swap]{dr}{G\times G} 
    \ar[bend left=30,""{name=AG,below,pos=0.65},phantom]{drrr} \ar[bend right=10,""{name=TGA},phantom]{drrr}
     \ar[Rightarrow,from=TGA,to=AG,bend right=20,"\mu_G"]
   && \cat{C} \ar[bend left]{dr}{G} \\
   &  \cat{C'}\times\cat{C'} \ar[bend left,swap]{dl}{L\times L} \ar{rr}{\otimes}
    \ar[bend left=50,""{name=AL},phantom]{dr} \ar[bend right=50,""{name=TLA},phantom]{dr}
     \ar[Rightarrow, bend right=20, from= TLA, to=AL,"\mu_L"]
    && \cat{C'} \ar[bend left]{dl}{L} \\
   \cat{D}\times\cat{D} \ar[swap]{rr}{\otimes} && \cat{D}   
  \end{tikzcd}
  \; \equiv \;
  \begin{tikzcd}
    \cat{C}\times\cat{C} \ar[swap, bend right=50]{dd}{F\times F} \ar{rr}{\otimes}
    \ar[bend right=75,""{name=TFA,pos=0.45},phantom]{ddrr} \ar[bend left=50,""{name=AF,pos=0.4},phantom]{ddrr}
     \ar[Rightarrow,from=TFA,to=AF,"\mu_F",bend left=30]
   && \cat{C} \ar[swap,bend right=50]{dd}{F} \ar[bend left]{dr}{G}
    \ar[bend right=30,""{name=F},phantom]{dd} \ar[bend left=90,""{name=GL},phantom]{dd}
     \ar[Rightarrow, from=F,to=GL,"\lambda"] \\
    &&& \cat{C'} \ar[bend left]{dl}{L} \\
   \cat{D}\times\cat{D} \ar[swap]{rr}{\otimes} && \cat{D}
  \end{tikzcd}
 \end{equation}
 and similarly for $\chi$, in order to compute
 \begin{align*}
  &\begin{tikzcd}[ampersand replacement=\&, row sep=large]
   \cat{C}\times\cat{C} \ar[bend left,swap]{dr}{G\times G} \ar[swap,bend right]{dd}{F\times F} \ar{rr}{\otimes} 
    \ar[bend right=5,""{name=GGT,pos=0.58},phantom]{drrr} \ar[bend left=25,""{name=TG,pos=0.7},phantom]{drrr} 
     \ar[Rightarrow, bend right=20,from=GGT,to=TG,"\mu_G","\cong"{swap}]
    \ar[bend right=10,""{name=FF,pos=0.45},phantom]{dd} \ar[bend left=60,""{name=GGLL,pos=0.42},phantom]{dd} 
     \ar[Rightarrow,bend left=20,from=FF,to=GGLL,"\lambda\times\lambda"]
    \&\& \cat{C} \ar[bend left]{dr}{G} \\
   \& \cat{C'}\times\cat{C'} \ar[swap,bend right,""{name=LL,right}]{dl}{L\times L} \ar[bend left, pos=0.15,""{name=XX,left,pos=0.5}]{dl}{X\times X}
    \ar[Rightarrow,from=LL,to=XX,"u\times u"]
   \ar[bend right=45,""{name=XXT,pos=0.45},phantom]{dr} \ar[bend left=90,""{name=TX,pos=0.75},phantom]{dr} 
    \ar[Rightarrow, bend right=15, from=XXT,to=TX, "\mu_X"]
   \ar{rr}{\otimes} \&\& \cat{C'} \ar[bend left,""{name=X,below}]{dl}{X} \\
   \cat{D}\times\cat{D}  \ar[swap]{rr}{\otimes} \&\& \cat{D}
  \end{tikzcd} \\[4pt]
  \equiv \quad
  &\begin{tikzcd}[ampersand replacement=\&, row sep=large]
   \cat{C}\times\cat{C} \ar[bend left,swap]{dr}{G\times G} \ar[swap,bend right]{dd}{F\times F} \ar{rr}{\otimes}   
    \ar[bend right=5,""{name=GGT,pos=0.58},phantom]{drrr} \ar[bend left=25,""{name=TG,pos=0.7},phantom]{drrr} 
     \ar[Rightarrow, bend right=20,from=GGT,to=TG,"\mu_G","\cong"{swap}]
    \ar[bend right=10,""{name=FF},phantom]{dd} \ar[bend left=65,""{name=GGXX,pos=0.5},phantom]{dd} 
     \ar[Rightarrow,bend right=10,from=FF,to=GGXX,"\chi\times\chi"]
    \&\& \cat{C} \ar[bend left]{dr}{G} \\
   \&  \cat{C'}\times\cat{C'} \ar[bend left, pos=0.15]{dl}{X\times X} \ar{rr}{\otimes} 
    \ar[bend right=45,""{name=XXT,pos=0.45},phantom]{dr} \ar[bend left=90,""{name=TX,pos=0.75},phantom]{dr} 
     \ar[Rightarrow, bend right=15, from=XXT,to=TX, "\mu_X"]
    \&\& \cat{C'} \ar[bend left,""{name=X,below}]{dl}{X}  \\
   \cat{D}\times\cat{D}  \ar[swap]{rr}{\otimes} \&\& \cat{D}
  \end{tikzcd}\\[4pt]
   \equiv\quad 
   &\begin{tikzcd}[ampersand replacement=\&, row sep=large]
   \cat{C}\times\cat{C} \ar[swap,bend right]{dd}{F\times F} \ar{rr}{\otimes} 
    \ar[bend right=60,""{name=FFT,pos=0.4},phantom]{ddrr} \ar[bend left=40,""{name=TF,pos=0.48},phantom]{ddrr} 
     \ar[Rightarrow,bend left=20,from=FFT,to=TF,"\mu_F",swap] 
    \&\& \cat{C} \ar[bend right,""{name=F,above},swap]{dd}{F} \ar[bend left]{dr}{G} 
    \ar[bend right=10,""{name=F},phantom]{dd} \ar[bend left=65,""{name=GX,pos=0.5},phantom]{dd} 
     \ar[Rightarrow,bend right=10,from=F,to=GX,"\chi"] \\
   \& \phantom{ \cat{C'}\times\cat{C'} } \&\&  \cat{C'} \ar[bend left,""{name=X,below}]{dl}{X} \\
   \cat{D}\times\cat{D}  \ar[swap]{rr}{\otimes} \&\& \cat{D}
  \end{tikzcd} \\[4pt]
  \equiv \quad
  &\begin{tikzcd}[ampersand replacement=\&, row sep=large]
   \cat{C}\times\cat{C} \ar[swap,bend right]{dd}{F\times F} \ar{rr}{\otimes} 
    \ar[bend right=60,""{name=FFT,pos=0.4},phantom]{ddrr} \ar[bend left=40,""{name=TF,pos=0.48},phantom]{ddrr} 
     \ar[Rightarrow,bend left=20,from=FFT,to=TF,"\mu_F",swap] 
    \&\& \cat{C} \ar[bend right,swap]{dd}{F} \ar[bend left]{dr}{G}
    \ar[bend right=10,""{name=F},phantom]{dd} \ar[bend left=60,""{name=GL,pos=0.45},phantom]{dd} 
     \ar[Rightarrow,bend left=20,from=F,to=GL,"\lambda"] \\
   \& \phantom{ \cat{C'}\times\cat{C'} } \&\& \cat{C'} \ar[bend left,""{name=X,above}]{dl}{X} \ar[bend right,swap,""{name=L,below}]{dl}[pos=0.6]{L}
    \ar[Rightarrow,from=L,to=X,"u"] \\
   \cat{D}\times\cat{D}  \ar[swap]{rr}{\otimes} \&\& \cat{D}
  \end{tikzcd}\\[4pt]
  \equiv\quad 
  &\begin{tikzcd}[ampersand replacement=\&, row sep=large]
   \cat{C}\times\cat{C} \ar[bend left,swap]{dr}{G\times G} \ar[swap,bend right]{dd}{F\times F} \ar{rr}{\otimes}
    \ar[bend right=5,""{name=GGT,pos=0.58},phantom]{drrr} \ar[bend left=25,""{name=TG,pos=0.7},phantom]{drrr} 
     \ar[Rightarrow, bend right=20,from=GGT,to=TG,"\mu_G","\cong"{swap}]
    \ar[bend right=10,""{name=FF,pos=0.45},phantom]{dd} \ar[bend left=60,""{name=GGLL,pos=0.42},phantom]{dd} 
     \ar[Rightarrow,bend left=20,from=FF,to=GGLL,"\lambda\times\lambda"]
    \&\& \cat{C} \ar[bend left]{dr}{G} \\
   \& \cat{C'}\times\cat{C'} \ar[swap,bend right,""{name=LL,above}]{dl}{L\times L} \ar{rr}{\otimes} 
    \ar[bend right=90,""{name=LLT,pos=0.33},phantom]{dr} \ar[bend left=45,""{name=TL,pos=0.52},phantom]{dr} 
     \ar[Rightarrow, bend left=15, from=LLT,to=TL, "\mu_L",swap]
    \& \& \cat{C'} \ar[bend left,""{name=X,above}]{dl}{X} \ar[bend right,swap,""{name=L,below}]{dl}[pos=0.8]{L} 
    \ar[Rightarrow,from=L,to=X,"u"] \\
   \cat{D}\times\cat{D}  \ar[swap]{rr}{\otimes} \&\& \cat{D}
  \end{tikzcd}
 \end{align*}
 The natural transformation \eqref{lambdaepi} is epic, so that $\lambda\times\lambda$, whiskered by $\cat{D}\times\cat{D}\to\cat{D}$, can be canceled. $\mu_G$ is an isomorphism, so that it can be canceled as well. Finally $G\times G$ is essentially surjective, and therefore pre-whiskering by it can also be canceled. We are then left with \eqref{multu}.
 
Now suppose that $\eta'_L$ and $\mu'_L$ give another monoidal structure on $L$. For $\lambda$ to be monoidal, the equations \eqref{unitlambda} and \eqref{multlambda} need to be satisfied. But now by~\eqref{unitlambda} and the invertibility of $\eta_G$, we get $\eta'_L = \eta_L$. Similarly, by~\eqref{multlambda} and the fact that $\mu_G$ is an isomorphism, $\lambda\otimes\lambda$ is epic, and $G\times G$ is essentially surjective, we conclude that $\mu'_L=\mu_L$.
\end{proof}

It may help to visualize these equations three-dimensionally, by interpreting every rewriting step as a globular 3-cell, and whiskering and composing these 3-cells so as to form a 3-dimensional pasting diagram. This way,~\eqref{multlambda} becomes a full cylinder, with the two caps formed by $\lambda\times\lambda$ and $\lambda$, and with the three multiplications forming the side surface. The same is true for equation~\eqref{unitlambda}, but with the $\lambda\times\lambda$ cap collapsed to a single point, so that one obtains a cone with $\lambda$ on the base.

\section{Convex combinations of metric spaces}
\label{ccms}

In this section, we will establish that $\cat{Met}$ is a pseudoalgebra for the simplex operad, which means that there is a notion of ``convex combination'' of metric spaces which behaves very similarly to convex combinations in a vector space. The $n$-fold uniform convex combination of a metric space $X$ with itself is the power $X^n$ from Section~\ref{powerfunsec}. This motivates the definition of the latter, and in particular the rescaling of the metric by $\tfrac{1}{n}$, which may otherwise seem unnatural. Very informally, by Theorem~\ref{colimitcmet}, we can therefore think of the Wasserstein space $PX$ as a uniform infinitary convex combination of $X$ with itself.

\subsection{The simplex operad}
\label{operad}

Standard simplices, or equivalently spaces of probability measures on finite sets, or also equivalently convex combinations of arbitrary finite arity, naturally form an operad.

\begin{deph}[{e.g.~\cites{leinstercafe,bf}}]
 The \emph{simplex operad} $\Delta$, or \emph{convex combination operad} has a set of operations of arity $n$ given by the $(n-1)$-simplex\footnote{Usually $\Delta^n$ stands for the $n$-simplex of dimension $n$, but in this context it is less confusing to index the standard simplices by the number of vertices rather than by dimension.},
 \begin{equation}
  \Delta^n:=\left\{(\lambda_1,\dots,\lambda_n) \in\R^n \:\bigg|\: \lambda_i \ge 0, \sum_i\lambda_i=1\right\}
 \end{equation}
 The unit of the operad is the unique element $1\in\Delta^1$.  The composition is defined as
 \begin{align*}
  \Delta^n \times \Delta^{m_1}\times\ldots\times \Delta^{m_n} & \longrightarrow \Delta^{m_1 + \ldots + m_n} \\
  ( \nu, \lambda_1,\dots,\lambda_n) &\longmapsto  (\nu_1\lambda_{11},\dots,\nu_1\lambda_{1m_1},\dots,\nu_n\lambda_{n1},\dots,\nu_n\lambda_{nm_n}).
 \end{align*}
 The symmetry of the operad is given by permutations of the vertices of the simplex: for $\sigma\in S_n$,
 \begin{align*}
  S_n \times \Delta^n &\longrightarrow \Delta^n \\
  (\sigma,\lambda) &\longmapsto (\lambda_{\sigma(1)},\dots,\lambda_{\sigma(n)}).
 \end{align*}
\end{deph}

For example, every convex space (Definition~\ref{csdef}) is an algebra of $\Delta$ in $\cat{Set}$. It can be shown that $\Delta$ is the operad generated by binary operations $c_\lambda$ with $\lambda\in[0,1]$ subject to the relations encoded by the parametric symmetry and associativity of Definition~\ref{csdef}. In other words, the concept of $\Delta$-algebra in $\cat{Set}$ is a relaxation of the notion of convex space without the requirements~\ref{csdef}\ref{csunit}--\ref{csidempo}.

\subsection{Pseudometric spaces are a pseudoalgebra}
\label{pmet}

We can turn $\cat{Met}$ into a \emph{pseudoalgebra} (or \emph{weak algebra}) in $\cat{Cat}$ of the simplex operad. To do so, it is more convenient to work with a similar category $\cat{PMet}$, whose objects are \emph{pseudometric spaces}, i.e.~sets with a distance function satisfying $d(x,x) = 0$, symmetry, and the triangle inequality, but where $d(x,y) = 0$ may not imply $x = y$. As morphisms, we again choose short maps. In the following, we define a $\Delta$-pseudoalgebra structure on $\cat{PMet}$, where it is more explicit. Below, we will use this to turn $\cat{Met}$ into a $\Delta$-pseudoalgebra as well. Since all operations turn out to map the full subcategory $\cat{CMet}$ to itself, we also have a pseudoalgebra structure on $\cat{CMet}$.

In the following, in contrast to the main text, we use subscript notation $X_i$ to refer to members of a sequence or family of spaces.

\begin{deph}
For any $n$ and $\lambda\in\Delta^n$, the functor
\[
	k_\lambda:\cat{PMet}^n\to \cat{PMet}
\]
takes a tuple of pseudometric spaces $(X_1,\dots,X_n)$ to the pseudometric space $k_\lambda(X_1,\dots,X_n)\in \cat{PMet}$, where:
 \begin{itemize}
  \item The underlying set is the Cartesian product $X_1\times \dots \times X_n$;
  \item The pseudometric is given by the convex combination of the pseudometrics,
  \begin{equation}\label{algmetric}
   d\big( (x_1,\dots,x_n), (y_1,\dots,y_n) \big) := \lambda_1\,d(x_1,y_1) + \dots + \lambda_n\,d(x_n,y_n).
  \end{equation}
 \end{itemize}
\label{PMetpsuedoalg}
\end{deph}

This strongly suggests that $\cat{PMet}$ is indeed a pseudoalgebra of $\Delta$ in a canonical way. For the technical details, we use~\cite[Definition~4.1]{weber} as the definition of pseudoalgebra. First,~\eqref{algmetric} shows why it is important to work with pseudometric spaces, since this metric is guaranteed to be nondegenerate only when $\lambda$ has full support; thus setting up the same structure on $\cat{Met}$ directly would be more cumbersome, since then the underlying set would also have to vary with $\lambda$.

In order to construct the necessary coherence maps and prove the required equations, it is most convenient to use the fact that the Cartesian products make $\cat{Set}$ into a symmetric monoidal category, which is the same thing as a pseudoalgebra of the terminal operad $\cat{Comm}$~\cite[Example~4.5]{weber}. (This gives an unbiased definition of symmetric monoidal category, in the spirit of~\cite[Definition~3.1.1]{leinster}.) Thanks to the unique operad morphism $\Delta\to\cat{Comm}$, the category $\cat{Set}$ becomes a pseudoalgebra of $\Delta$ as well, where the action by $\lambda\in\Delta^n$ is given by $(X_1,\ldots,X_n)\mapsto X_1\times\ldots\times X_n$. This corresponds to the $k_\lambda$ from Definition~\ref{PMetpsuedoalg} under the forgetful functor $\cat{PMet}\to\cat{Set}$. Since the latter functor is faithful, there is at most one way to lift the coherence isomorphisms to $\cat{PMet}$, where the required equations then automatically hold. Proving that such liftings exist boils down to showing that the coherence isomorphisms are isometries, but this is straightforward to see from the definition~\eqref{algmetric}.

It remains to transport this pseudoalgebra structure to $\cat{Met}$. To this end, we first consider the category $\cat{PMet}_\sim$, which is the quotient of $\cat{PMet}$ where $f,g:X\to Y$ are considered equivalent if $d(f(x),g(x))=0$ for all $x\in X$. Considering a metric space as a pseudometric space gives a fully faithful functor $\cat{Met}\to\cat{PMet}_\sim$. By identifying all points that have distance zero, every pseudometric space has a metric quotient to which it is isomorphic in $\cat{PMet}_\sim$, using the axiom of choice to construct a section of the quotient map. Therefore the functor $\cat{PMet}_\sim\to\cat{Met}$ is an equivalence of categories, and it is enough to construct a $\Delta$-pseudoalgebra structure on $\cat{PMet}_\sim$. But this follows immediately from the fact that the functors $k_\lambda$ from Definition~\ref{PMetpsuedoalg} respect the equivalence relation on morphisms and therefore descend to functors $\cat{PMet}_\sim\to\cat{PMet}_\sim$.

In summary, we have turned the category $\cat{Met}$ into a pseudoalgebra of the simplex operad $\Delta$. The $n$-ary convex combination with weights $\lambda\in\Delta^n$ of metric spaces $X_1,\ldots,X_n$ is the metric space with underlying set $X_1\times\ldots\times X_n$, but where those $X_i$ with $\lambda_i = 0$ are omitted, and metric~\eqref{algmetric}. The power $X^n$ is the special case where $\lambda=(\tfrac{1}{n},\ldots,\tfrac{1}{n})$ is the uniform distribution, since then one gets from~\ref{algmetric},
\begin{equation}
   d\big( (x_1,\dots,x_n), (y_1,\dots,y_n) \big) = \dfrac{1}{n} \sum_{i=1}^n d(x_i,y_i) .
  \end{equation}
which is exactly the metric on the power $X^n$ defined by~\ref{pfmetric}.

Strangely, the category of finite probability spaces and measure-preserving maps, $\cat{FinProb}$, is also a pseudoalgebra of $\Delta$~\cite[Appendix~B]{bf}, but in a very different way, where a convex combination of spaces is based on the disjoint union of sets with the corresponding convex combination of probability measures, supported on these disjoint sets. At present we do not know whether to regard this as coincidental or deep.

\subsection{Internal algebras and the microcosm principle}
\label{microcosm}

The concept of monoid makes sense internally to any monoidal category. Similarly, commutative monoids can be considered internally to any symmetric monoidal category. More in general, there are situations where one expects an internal structure to be naturally defined within a categorified version of that same structure. This phenomenon has been called ``microcosm principle'' by Baez and Dolan \cite{hda3}, and they made this statement precise in terms of operads: if a category $\cat{C}$ is equipped with a pseudoalgebra structure for an operad $O$, then an internal algebra in $\cat{C}$ is specified by a lax morphism of pseudoalgebras $\cat{1}\to\cat{C}$. In the following, we will show that the internal $\Delta$-algebras in $\cat{CMet}$ are closely related to the $P$-algebras of Section~\ref{Palgssec}.

In our case, we have the simplex operad $\Delta$, and the pseudoalgebra structure on $\cat{Met}$, or equivalently on $\cat{PMet}_\sim$. The terminal category $\cat{1}$ is also trivially an algebra for $\Delta$. A lax morphism between these pseudoalgebras is then a functor $A:\cat{1}\to\cat{PMet}_\sim$, selecting an object $A\in\cat{PMet}_\sim$, and for each $n$ and $\lambda\in\Delta^n$ a natural transformation
\begin{equation}
 \begin{tikzcd}
  \cat{1}^n \ar[swap]{d}{!} \ar{r}{A^{\times n}} 
   \ar[bend left,""{name=ANK},phantom]{dr} \ar[bend right,""{name=A},phantom]{dr} 
    \ar[Rightarrow,from=ANK,to=A,"c_\lambda"]
   & \cat{PMet}_\sim^n \ar{d}{k_\lambda}  \\
  \cat{1} \ar[swap]{r}{A} & \cat{PMet}_\sim
 \end{tikzcd}
\end{equation}
or equivalently a morphism $c_\lambda : k_\lambda(A,\ldots,A)\to A$, satisfying certain coherence conditions which we specify below. Before doing so, it is instructive to look at what it means for $c_\lambda$ to be short: the metric on $k_\lambda(A,\ldots,A)$ is given by
\begin{equation}
 d \big( (x_1,\dots, x_n) , (y_1,\dots,y_n) \big) = \sum_i \lambda_i \, d(x_i,y_i).
\end{equation}
Thus shortness of $c_\lambda$ means that
\begin{equation}
 d \big( c_\lambda(x_1,\dots, x_n) , c_\lambda(y_1,\dots,y_n) \big) \le \sum_i \lambda_i \, d(x_i,y_i),
\end{equation}
which is the metric compatibility inequality from Section~\ref{Palgssec} in the generalized form~\eqref{metcompatgeneral}.

Following~\cite[Definition~4.10]{weber}, the $c_\lambda$'s have to satisfy the following conditions in order to define a lax morphism:
\begin{itemize}
 \item Unit condition. For $1\in\Delta^1$, the composite
 \begin{equation}
  \begin{tikzcd}[column sep=huge]
	\cat{1} \ar[bend right=70,swap]{r}{A} \ar["k_1(A)",near start]{r} \ar[bend left=70]{r}{A}  
	 \ar[bend left=60,swap,""{name=UP},phantom]{r} \ar[bend left=10,swap,""{name=UK,pos=0.635},phantom]{r} 
	 \ar[bend right=10,swap,""{name=DK,pos=0.62},phantom]{r} \ar[bend right=60,swap,""{name=DN},phantom]{r} 
	  \ar[Rightarrow,from=UP,to=UK,"\iota"] \ar[Rightarrow,from=DK,to=DN,"c_1"]
	 & \cat{PMet}_\sim 
  \end{tikzcd}
 \end{equation}
must be equal to the identity transformation of $A$, where $\iota$ is one of the coherences of the pseudoalgebra structure of $\cat{PMet}$. So explicitly, this means that the map $c_1 : k_1(A)\to A$ must be inverse to the canonical isomorphism $A\to k_1(A)$. This is the same as the unit condition for convex spaces when defined as algebras of the convex combinations monad.
 
 \item Composition condition. We consider arities $n$ with $\lambda\in\Delta_n$ and $m_1,\dots,m_n$ with $\mu_i\in\Delta_{m_i}$, and write $m:=\sum_i m_i$. The composite 2-cell
  \begin{equation}
   \begin{tikzcd}[row sep=large,column sep=large]
    & \cat{PMet}_\sim^m \ar{r}{\prod_i k_{\mu_i}} 
     & \cat{PMet}_\sim^n \ar{dd}{k_\lambda}  \\
    \cat{1}^m \ar{ur}{A^{\times m}} \ar{dr} \ar{r} 
     \ar[bend left=14, ""{name=AMPK,pos=0.51},phantom]{urr} \ar[bend right=10, ""{name=AN,pos=0.46},phantom]{urr} 
      \ar[Rightarrow, from=AMPK,to=AN,"\prod_i c_{\mu_i}"]
     \ar[bend left=50,""{name=ONEN,pos=0.62},phantom]{dr} \ar[bend left=10,""{name=UNIQ,pos=0.55},phantom]{dr}
      \ar[Rightarrow, from=ONEN,to=UNIQ]
     & \cat{1}^n \ar[swap]{ur}{A^{\times n}}  \ar{d} 
      \ar[bend left=60,""{name=ANK,pos=0.45},phantom]{dr} \ar[bend right=20,""{name=A,pos=0.4},phantom]{dr} 
       \ar[Rightarrow, from=ANK,to=A,"c_\lambda"] \\
    & \cat{1} \ar{r}{A} & \cat{PMet}_\sim 
   \end{tikzcd}
  \end{equation}
has to be equal to the composite
  \begin{equation}
   \begin{tikzcd}[row sep=large,column sep=large]
    & \cat{PMet}^m_\sim \ar{r}{\prod_i k_{\mu_i}} \ar[swap]{ddr}{k_{\lambda\circ\mu}} 
     \ar[bend left=60,""{name=PKK,pos=0.4},phantom]{ddr} \ar[bend left=10,""{name=K,pos=0.45},phantom]{ddr}
      \ar[Rightarrow, from=PKK, to=K]
     & \cat{PMet}^n_\sim \ar{dd}{k_\lambda}  \\
    \cat{1}^m \ar{ur}{A^{\times m}} \ar{dr} 
     \ar[bend left=48,""{name=AMK,pos=0.386},phantom]{drr} \ar[""{name=A},phantom]{drr} 
      \ar[Rightarrow, from=AMK,to=A, "c_{\lambda\circ\mu}",swap] \\
    & \cat{1} \ar{r}{A} & \cat{PMet}_\sim
   \end{tikzcd}
  \end{equation}
  where all unlabelled morphisms are the obvious ones. This amounts to commutativity of
  \begin{equation}
   \begin{tikzcd}
    k_\lambda(k_{\mu_1}(A,\ldots,A),\ldots,k_{\mu_n}(A,\ldots,A)) \ar{d} \ar{rr}{k_\lambda(c_{\mu_1},\ldots,c_{\mu_n})} && k_\lambda(A,\ldots,A) \ar{d}{c_\lambda} \\
    k_{\lambda\circ\mu}(A,\ldots,A) \ar{rr}{c_{\lambda\circ\mu}} && A
   \end{tikzcd}
  \end{equation}
  
 \item Equivariance condition. Since all symmetries of $\cat{1}$ are trivial, this is basically only an invariance condition. Let $\sigma \in  S_n$. 
The composite 2-cell
 \begin{equation}
  \begin{tikzcd}[column sep=tiny]
   && \cat{PMet}_\sim^n  \ar{drr}{\sigma} \ar{ddr}[swap]{k_\lambda} 
    \ar[bend left=70,""{name=SK,pos=0.67},phantom]{ddr} \ar[bend left=20,""{name=K,pos=0.6},phantom]{ddr}
     \ar[Rightarrow, from=SK, to=K]\\
   \cat{1}^n  \ar{urr}{A^{\times n}} \ar[swap]{dr}{!} 
    \ar[bend left=50,""{name=ANK,pos=0.4},phantom]{drrr} \ar[bend right=8,""{name=A},phantom]{drrr} 
     \ar[Rightarrow,from=ANK,to=A,"c_\lambda",swap]
    && |[alias=Z]| && \cat{PMet}_\sim^n \ar{dl}{k_{\sigma(\lambda)}} \\
   & \cat{1} \ar[swap,""{name=D}]{rr}{A} && \cat{PMet}_\sim
  \end{tikzcd}
 \end{equation}
 has to be equal to the composite
 \begin{equation}
  \begin{tikzcd}[column sep=tiny]
   && \cat{PMet}_\sim^n  \ar{drr}{\sigma}  \\
   \cat{1}^n  \ar{urr}{A^{\times n}} \ar[swap]{dr}{!} \ar{rr}{\sigma} 
    \ar[bend left=33,""{name=ANS,pos=0.4},phantom]{rrrr} \ar[bend left=13,""{name=SAN,pos=0.41},phantom]{rrrr} 
     \ar[equal,from=ANS,to=SAN]
    \ar[bend left=90,""{name=S,pos=0.55},phantom]{dr} \ar[bend left=15,""{name=U},phantom]{dr}
     \ar[Rightarrow, from=S, to=U]
    && \cat{1}^n \ar{rr}{A^{\times n}} \ar{dl}  
     \ar[bend left=50,""{name=AN,pos=0.7},phantom]{dr} \ar[bend right=45,""{name=KA},phantom]{dr} 
      \ar[Rightarrow,from=AN,to=KA,"c_{\sigma(\lambda)}",swap]
     && \cat{PMet}_\sim^n \ar{dl}{k_{\sigma(\lambda)}} \\
   & \cat{1} \ar[swap]{rr}{A} && \cat{PMet}_\sim
  \end{tikzcd}
 \end{equation}
where all unlabelled morphisms are the obvious ones. This amounts to commutativity of the following equivariance diagram,
 \begin{equation}
  \begin{tikzcd}
   k_{\sigma(\lambda)}(A,\ldots,A) \ar{d} \ar{r}{c_{\sigma(\lambda)}} & A \idar{d} \\
   k_\lambda(A,\ldots,A) \ar{r}{c_\lambda} & A
  \end{tikzcd}
 \end{equation}
 where the left arrow corresponds to the coherence isomorphism associated to the pseudoalgebra structure of $\cat{PMet}_\sim$, i.e.~to permutation of the factors of the underlying cartesian product.
\end{itemize}

Taken together, these conditions state precisely that the underlying set of $A$ must be an algebra of the simplex operad, with the additional condition that the maps $c_\lambda:k_\lambda(A,\ldots,A)\to A$ must be short. The latter is exactly the metric compatibility inequality in the generalized form~\eqref{metcompatgeneral}. So by Theorem~\ref{Palgthm}, there is a fully faithful functor from $P$-algebras to internal $\Delta$-algebras in $\cat{PMet}_\sim$. However, this functor is not essentially surjective, since the unitality and idempotency conditions~\ref{csdef}\ref{csunit}--\ref{csidempo} cannot be expressed in operadic terms. For example, we can choose any fixed $y\in A$ and define $c_\lambda(x_1,\ldots,x_n):=y$ for all $\lambda\in\Delta^n$ of arity $n\geq 2$. This defines an internal $\Delta$-algebra in $\cat{PMet}_\sim$ which is not in the essential image of the forgetful functor from $\cat{CMet}^P$ if $|A|\geq 2$.

In order to improve on this and to obtain a complete characterization of $P$-algebras as internal algebras, it may therefore be necessary to consider $\cat{PMet}_\sim$ as a lax algebra of a suitable 2-monad. We have not done this yet.

\bibliography{catprob}

\end{document}